\documentclass[a4paper,10pt,draft]{article}

\usepackage[latin9]{inputenc}
\usepackage{amsfonts}
\usepackage{amsmath}
\usepackage{amssymb}
\usepackage{amsthm}
\usepackage[english]{babel}
\usepackage{caption}
\usepackage{fontenc}
\usepackage{graphicx}
\usepackage{fancyhdr}
\usepackage{etoolbox}
\usepackage{url}
\usepackage{verbatim}
\usepackage{color}
\usepackage{xcolor}
\usepackage{mdframed}
\usepackage[sectionbib,square]{natbib}

\makeatletter
\patchcmd{\thebibliography}{%
   \chapter*{\bibname}\@mkboth{\MakeUppercase\bibname}{\MakeUppercase\bibname}}{%
      \section*{References}}{}{}
\makeatother

\title{Strong convergence of a positive preserving drift-implicit Euler scheme for the fixed delay CIR process}
\author{Federico Flore\thanks{Dipartimento di Economia Aziendale - Universit\`{a} degli Studi Roma Tre, Via Silvio D'Amico 77, 00145 Rome, Italy.} \and Giovanna
Nappo\thanks{Dipartimento di Matematica - Universit\`{a} di Roma {``}Sapienza", Piazzale A. Moro 5, 00185 Rome, Italy.}}

\theoremstyle{plain}
\newtheorem{theorem}{Theorem}[section]
\newtheorem{lemma}[theorem]{Lemma}
\newtheorem{corollary}[theorem]{Corollary}
\newtheorem{proposition}[theorem]{Proposition}

\theoremstyle{definition}

\newtheorem{assumptions}[theorem]{Assumptions}

\theoremstyle{plain}
\newtheorem{remark}{Remark}

\begin{document}
\maketitle
\begin{abstract}
In this paper, we consider a fixed delay CIR process on the regime where it doesn't hit zero, the aim is to determine a positive preserving implicit Euler Scheme. On a time grid with stepsize $\Delta$ our scheme extends the scheme  proposed in~\cite{alfonsi2005discretization} for the classical CIR model.   Furthermore,  we consider its piecewise linear interpolation, and, under suitable conditions, we establish  the order of strong convergence in the uniform norm, thus extending the results in~\cite{dereich2011euler}.
\end{abstract}
\section{Introduction}\label{sec:introEA}
Stochastic differential equations (SDEs)  and stochastic delay differential equations (SDDEs) arise naturally in the modeling of realistic physical, biological systems, as well as financial and actuarial systems. In general, though one can prove existence, uniqueness and other properties of the solutions, no  explicit solutions of these equations are available,  and numerical approximation schemes are needed.

In this paper, our aim is to present a positive preserving discretization scheme strongly convergent to
 a fixed delay CIR process $X^{(b)}(t)$, defined by
\begin{align}
 X^{(b)}(t)&=X_0(t) \qquad t_0-\tau\leq t \leq t_0,\label{eq:CIRdelayX-InitialSegment} \\
 dX^{(b)}(t)&=[a(\gamma(t)-X^{(b)}(t))+bX^{(b)}(t-\tau)]dt+\sigma\sqrt{X^{(b)}(t)}dW(t), \quad t>t_0\label{eq:CIRdelayX}
\end{align}
where the parameters $a$ and $\sigma$ are positive constants, the parameter $b$ is a nonnegative constant, $\gamma(t)$ is a positive deterministic measurable function, and the initial segment process $X_0(t)$, $t\in [t_0-\tau,t_0]$,  is a random positive process.
In~\cite{FloreNappo_Delay}, the authors prove that, under suitable hypotheses (see Assumptions~\ref{ass:hp_EA}), Eq.~\eqref{eq:CIRdelayX} with initial segment process~\eqref{eq:CIRdelayX-InitialSegment}   admits a unique nonnegative solution, and, under the generalized Feller condition
 \begin{equation*}\label{eq:FellerCond}
\sigma^2\leq 2 a \gamma(t)\quad\text{for all $t\geq t_0$},
\end{equation*}
 the solution is positive,
 i.e.,~$\mathbb{P}\left(X^{(b)}(t)>0\text{ for $t\geq t_0$}\right)=1$.
Note that when $b=0$ and the function $\gamma(t)=\gamma$ is constant then the process $X^{(0)}(t)$ coincide with the classical CIR model. \\

 The  fixed delay CIR process, as well as the classical CIR process, can be used to model random intensity process for Cox processes, and therefore to model default/death random times
 $$\mathcal{T}= \inf\{s\geq t_0: \; \textstyle{\int_{t_0}^s X^{(b)}(u)\,du \geq E}\},$$
  where $E$ is an exponential random variable, with parameter $1$, independent of the process $X^{(b)}(\cdot)$, and then  $\mathbb{P}[\mathcal{T}> T|\mathcal{F}_t]=\mathbb{E}[e^{-\int_t^T X^{(b)}(u)\,du}|\mathcal{F}_t]$, on $\{   \mathcal{T}>t\}$.
   To this end the property that  $X^{(b)}(t)$ is positive for all $t\geq t_0$,   is crucial. The latter positivity  property is also fundamental to use it as a model of a random volatility process. Clearly it can also be used as a model of a random interest rate process under the risk neutral probability measure (though in this case the positivity property is not crucial);  also in this case it is important to compute~$\mathbb{E}[e^{-\int_t^T X^{(b)}(u)\,du}|\mathcal{F}_t]$, i.e., the zero coupon bond price.  It is well-known (see, e.g., \cite{lamberton_lapeyre}) that the  classical CIR model $X^{(0)}(t)$ is an affine process and this computation is explicitly determined as $\exp\{-\psi(t,T) -\phi(t,T) X^{(0)}(t)\}$, where $\psi(t,T)$ and $\phi(t,T)$ are deterministic positive functions. In~\cite{FloreNappo_Delay}, the authors prove that a similar result holds for the fixed delay process:
$$
\mathbb{E}[e^{-\int_t^T X^{(b)}(u)\,du}|\mathcal{F}_t]= e^{-a \int_t^T\gamma(u)\,\alpha(u,T)\, du - \alpha(t,T) X^{(b)}(t)- Y(t,T)},
$$
where $\alpha(t,T)$ is the (positive) solution of a deterministic delay differential equation, and
$$
Y(t,T)=\int_{[t_0-\tau, t] \cap [t-\tau, T-\tau]} b\,\alpha(u+\tau,T)\, X^{(b)}(u)\,du.
$$
It is then clear why it is important to find positive preserving approximations especially in the first two examples of applications: random default/death times and stochastic volatility.\\

The literature on weak and strong convergence of numerical approximation schemes for SDEs and SDDEs is huge. Limiting  to SDDEs we suggest, among others \cite{Kushner-book-1977,Kushner-2005,Kushner-2006,Kushner-book-2008,Kushner-2011}, \cite{Kuchler-Platen-2000, Kuchler-Platen-2002}, \cite{Mao-2003},   \cite{Chang-2008}, \cite{Wu-Mao-Chen-2009}, \cite{fischer2009moments}, \cite{Huang-2014}, \cite{Zhang-2018}, and the literature therein. Due to the diffusion coefficient $\sigma \sqrt{x}$, the fixed delay model does not fit the conditions needed in the quoted literature concerning strong convergence, though one could use the truncated Euler scheme analyzed by ~\cite{deelstra_delbaen} for the class of processes with stochastic drift term
satisfying the stochastic differential equation
\[dX(t)=(2\beta X(t)+\delta(t))dt+g(X(t))dW(t),\]
where $\beta$ is a negative real value, $\delta$ is a nonnegative adapted process such that
\[\int_{t_0}^t\delta(u)\,du<+\infty\quad\text{a.s.},\]
and $g(\cdot)$ is a H\"{o}lder continuous function vanishing at zero such that
\[|g(x)-g(y)|\leq K\sqrt{|x-y|}.\]
Indeed, setting
 \[\beta=-\frac{a}{2},\quad \delta(t)=  a\,\gamma(t)+b X(t-\tau),\quad\text{and}\quad g(x)=\sigma \sqrt{x},\]
 we recover the fixed delay CIR process, but, as usual in Euler truncated schemes, the approximating process assume negative values with positive probability, so that, this scheme is not positive preserving.
\\

To our knowledge strong convergence of positive preserving discretization schemes for SDDE with such a kind of diffusion coefficient have not been analyzed in  the literature, while this is the case for some classes of SDE:

In~\cite{bossy2007efficient} and~\cite{berkaoui2008euler}, the authors consider  the following stochastic differential equation
\[
dX(t)=b(X(t))dt+\sigma|X(t)|^\alpha dW(t),\\
\]
where $b(\cdot)$ is a Lipschitz function such that $b(0)>0$, $\sigma$ is a positive constant, $\alpha\in[\frac{1}{2},1)$ and the initial value~$X(t_0)=x\geq 0$, and study a symmetrized Euler scheme defined by
\[\begin{cases}
\overline{X}_0&=X(t_0),\phantom{\frac{.}{|}}\\
\overline{X}(t)&=\big|\overline{X}(t_{k})+b(\overline{X}(t_{k}))(t-t_k) +\sigma \overline{X}^\alpha(t_{k}) \big(W(t)-W(t_{k})\big)\big|, \quad \text{for}\quad t\in (t_k,t_{k+1}], \quad k\geq 0, \phantom{\frac{|}{.}}
\end{cases}\]
so that $\overline{X}(t)$ is a diffusion process with reflection. Though this scheme  preserves nonnegativity, it is not  positive preserving.\\
Setting $\Delta =t_k-t_{k-1}$, in~\cite{berkaoui2008euler}, the authors prove a strong convergence result, showing that, for all $p\geq 1$, there exists a positive constant $C(p)$ such that
\[
\mathbb{E}\left[\sup_{t_0\leq t\leq T}|X(t)-\overline{X}(t)|^{2p}\right]^{\frac{1}{2p}}\leq C(p)\sqrt{\Delta }.
\]
\cite{bossy2007efficient}  prove that the weak error is of order one in $\Delta$. In the particular case  $\alpha = \frac{1}{2}$, the previous results hold  under some further
conditions on $b(0)$ and $\sigma$.\\

A different method, known as splitting method, is analyzed by~\cite{MoroSchurz}. The authors prove that the method has a good convergence rate when the coefficients are sufficiently regular on the whole Euclidean space, and apply it numerically to various models, including the classical CIR model. \\

\cite{alfonsi2005discretization} has proposed a positive preserving drift implicit Euler scheme $\breve{X}(t)$ for the solution of the following  stochastic differential equation
\begin{equation*}\label{Eq:CIRAlfonsi2005}
dX(t)=(a-\kappa\,X(t))dt+\sigma\sqrt{X(t)}\,dW(t),\\
\end{equation*}
where $W$ denotes a standard Brownian motion, $a \geq 0$, $\kappa \in\mathbb{R}$, $\sigma > 0$, $t_0=0$, and $X(t_0)=x\geq 0$, which includes the classic CIR process $X(t)$.
 In this pioneer paper, the author considers a time horizon $T > 0$ and a regular stepsize $\Delta=t_k-t_{k-1}=\frac{T}{N}$;  under the strong Feller  condition  $2a > \sigma^2$, and when  $1+\kappa\frac{T}{N}>0$ the author proves that the weak convergence  rate of the drift implicit scheme is of order one in $\Delta$, while, for the strong convergence, he proves
\begin{equation*}\label{eq:SCAlfonsi2005}
\mathbb{E}\left[\sup_{t_0\leq t\leq T}|X(t)-\breve{X}(t)|\right]\leq \frac{C}{\sqrt{|\log(\Delta)|}}.
\end{equation*}
In the same paper, Alfonsi proposes also a different scheme, obtained via the implicit Euler scheme for the process~$Y(t)=\sqrt{X(t)}$, and  shows only numerically  that the scheme converges  very well.
 The approximation~$\widetilde{Y}(t)$  is defined implicitly by
\begin{equation}\label{eq:ESMAlfonsi-DNS}
\begin{cases}
\widetilde{Y}(t_0)&=\sqrt{X(t_0)},\\
\widetilde{Y}(t)&=\widetilde{Y}(t_k)+\left(\frac{a-\frac{\sigma^2}{4}}{2\widetilde{Y}(t)}-\frac{\kappa}{2}\,\widetilde{Y}(t)\right)(t-t_k)  +\frac{\sigma}{2}\big(W(t)-W(t_k)\big),\,\text{for $t\in(t_k,t_{k+1}]$},
\end{cases}\end{equation}
and, since $1+\kappa\frac{T}{2N}>0$, Eq.~\eqref{eq:ESMAlfonsi-DNS} has a unique solution $\widetilde{Y}(t)$ for $t\in(t_k,t_{k+1}]$, given by
\begin{equation*}\label{eq:dIES-Alfonsi}
\widetilde{Y}(t)= \frac{\widetilde{Y}(t_k)+ \frac{\sigma}{2}\big(W(t)-W(t_k)\big)+\sqrt{\big[\widetilde{Y}({t_k})+ \frac{\sigma}{2}\big(W(t)-W(t_k)\big)\big]^2+2\left(1+\frac{\kappa}{2}(t-t_k)\right)\left(a-\frac{\sigma^2}{4}\right)(t-t_k)}
}{2\big(1+\frac{\kappa}{2}\,(t-t_k)\big)}.
\end{equation*}
On the time grid $t_k$, the drift-implicit Euler scheme is defined by $y_k:=\widetilde{Y}(t_k)$. Consequently, the transformation~$x_k=y_k^2$ gives a positive approximation for the classical CIR model and the ``diffusive'' approximation is given by $\widetilde{X}(t) = \widetilde{Y}^2(t) $  for $t\in[t_0,T]$.

~\cite{dereich2011euler}, under the further assumption $\kappa>0$, prove a convergence result  for this scheme, using the approximation process $\widehat{X}(t)$ defined as the piecewise linear interpolation of $x_k$:\\
Under the strong Feller condition $\sigma^2<2a$, the authors show that
\begin{equation*}\label{eq:SCDNS}
\forall\; p\in\left[1,\frac{2a}{\sigma^2}\right)\quad \exists\; K_p \quad \text{s.t.}\quad \left(\mathbb{E}\left[\sup_{t_0\leq t\leq T}|X(t)-\widehat{X}(t)|^p\right]\right)^\frac{1}{p}\leq K_p\sqrt{\Delta|\log\left(\Delta\right)|}.
\end{equation*}

As noted in~\cite{alfonsi2013strong}, the result in~\cite{dereich2011euler} implies that, under the same conditions,
\begin{equation}\label{eq:SCAlfonsi2013-bis}
\textstyle{\left(\mathbb{E}\left[\sup_{t_0\leq t\leq T}|X(t)-\widetilde{X}(t)|^p\right]\right)^\frac{1}{p}}\leq K_p\,\sqrt{\Delta}.
\end{equation}
Indeed,  the ``diffusive'' approximation in~\cite{alfonsi2013strong} and the piecewise linear approximation considered in~\cite{dereich2011euler} share the same value $x_k=y_k^2$ on the  time grid $t_k$.

Moreover, ~\cite{alfonsi2013strong} shows that~\eqref{eq:SCAlfonsi2013-bis} still holds  when  $\kappa\leq 0$ and  $1+\kappa\frac{T}{2N}>0$, and proves that, under the more restrictive assumptions on the CIR parameters $\sigma^2<a$,
\begin{equation*}\label{eq:SCAlfonsi2013}
\forall \; p\in\left[1,\frac{4a}{3\sigma^2}\right) \quad \exists\; K_p\quad \text{s.t.} \quad \left(\mathbb{E}\left[\sup_{t_0\leq t\leq T}|X(t)-\widetilde{X}(t)|^p\right]\right)^\frac{1}{p}\leq K_p\,\Delta.
\end{equation*}
Furthermore the method used in~\cite{alfonsi2013strong} to get the convergence result may be applied to a larger  class of stochastic differential equations,  and in this sense it is more general than the strong convergence result in~\cite{dereich2011euler}.\\

Since we are  interested to positive preserving Euler-type methods,  we have generalized the  positive preserving scheme in~\cite{dereich2011euler} and~\cite{alfonsi2013strong} for the fixed delay CIR model considered in this paper. Hereunder we describe our generalization, the original drift implicit Euler scheme for the classical CIR model can be recovered by taking~$b=0$.

Consider the process $Y^{(b)}(t)=\sqrt{X^{(b)}(t)}$, which, by It\^o's formula, satisfies
\begin{align}
Y^{(b)}(t)&=\sqrt{X_0(t)}\qquad t_0-\tau\leq t\leq t_0,\label{eq:diffProcessY-InitialSegment}\\
dY^{(b)}(t)&= \left(\underline{a}(t)\frac{1}{Y^{(b)}(t)} - \overline{a} Y^{(b)}(t)\right) dt + \overline{b} \,\frac{\left(Y^{(b)}(t-\tau)\right)^2}{Y^{(b)}(t)}\, dt +   \overline{\sigma} \,dW(t),\quad t>t_0\label{eq:diffProcessY}
\end{align}
where
\begin{equation}\label{eq:notations}
\underline{a}(t)=\frac{4a\gamma(t)-\sigma^2}{8},\quad\overline{a}=\frac{a}{2},\quad\overline{b}=\frac{b}{2}\,\text{ and }\;
\overline{\sigma}=\frac{\sigma}{2}.\end{equation}
We consider a constant discretization step $\Delta$, and assume that  $\Delta=\frac{\tau}{N}$, for a fixed $N\in \mathbb{N}$,
 consequently,  for the time grid $t_{k}=t_0+k\Delta$ it holds
\begin{equation}\label{t-k-tau}
 t_{k+1}-\tau
=t_{k+1-N}.
\end{equation}
For notational convenience, in the sequel we will use also the symbol $\Delta t_k$ instead of $\Delta$.
\\
 The  ``diffusive'' paths approximation $\widetilde{Y}(t)$ of $Y^{(b)}(t)$ is implicitly defined by
\begin{equation}\label{eq:ESM}\begin{cases}
\widetilde{Y}(t)&=\widetilde{Y}_0(t),\,\text{for $t\in[t_0-\tau, t_0]$}\\
\widetilde{Y}(t)&=\widetilde{Y}({t_k})+\left(\underline{a}(t)\frac{1}{\widetilde{Y}(t)} - \overline{a} \widetilde{Y}(t)\right)(t-t_k)+\overline{b} \,\frac{\widetilde{Y}^2(t-\tau)}{\widetilde{Y}(t)}\, (t-t_k) +   \overline{\sigma} \,\Delta_k W(t),\,\text{for $t\in(t_k,t_{k+1}]$}
\end{cases}\end{equation}
where
\[\Delta_k W(t)=W(t)-W(t_k),\quad k\geq 0,\]
and  $\widetilde{Y}_0(t)$ is an approximation of~$\sqrt{X_0(t)}$ in~$[t_0-\tau,t_0]$ in a suitable sense.
\\

Since the parameters $\overline{a}$, $\overline{b}$ and $\overline{\sigma}$ are nonnegative,  and we assume that the function $\underline{a}(t)$ is positive, Eq.~\eqref{eq:ESM} has the unique positive solution given by, for $t\in(t_k,t_{k+1}]$,
\begin{equation*}\label{eq:dIES}
\widetilde{Y}(t)=\textstyle{ \tfrac{\widetilde{Y}(t_k)+ \overline{\sigma}\Delta_k W(t)+\sqrt{\big(\widetilde{Y}({t_k})+ \overline{\sigma}\Delta_k W(t)\big)^2+4\big(1+\overline{a}\,(t-t_k)\big)
\big(\underline{a}(t)+\overline{b}\,\widetilde{Y}^2(t-\tau)\big)(t-t_k)}}{2\big(1+\overline{a}\,(t-t_k)\big)}.
}
\end{equation*}

From now on, setting
\begin{equation}\label{eq:notations_yk}
y_k := \widetilde{Y} (t_k)\quad\text{for all $k\geq -N$},\end{equation}
 using \eqref{eq:ESM} and taking into account \eqref{t-k-tau}
 we get the following discrete time Euler implicit approximation scheme for the process $Y^{(b)}(t)$,
\begin{equation}\label{eq:ApproximationScheme}
y_{k+1}=y_k+\left(\underline{a}(t_{k+1})\frac{1}{y_{k+1}} - \overline{a} y_{k+1}\right)\Delta t_k+\overline{b} \,\frac{y^2_{k+1-N}}{y_{k+1}}\, \Delta t_k+   \overline{\sigma} \,\Delta W_k,\quad\text{with } \Delta W_k =W(t_{k+1})-W(t_k).
\end{equation}
By  the following position, we get the discrete time approximation scheme for the  process $X^{(b)}(t)$
\begin{align*}
&x_k:=y_k^2,\quad k\geq -N.
\label{eq:StartPosition2}
\end{align*}
Note that we do not necessarily assume $x_k =  X_0 (t_k) $ (or equivalently $y_k = \sqrt{X_0 (t_k)}$), for all $k$ such that $t_k \in [t_0- \tau, t_0]$.

In this paper we consider the piecewise linear approximation
\begin{equation}\label{eq:LinearInterpolationDNSScheme}\widehat{X}(t):=x_k+(t-t_k)\,\frac{x_{k+1}-x_k}{t_{k+1}-t_k}, \quad
t\in [t_k, t_{k+1}],
\end{equation}
and extend Proposition~$3.3$ and  The\-o\-rem~$1.1$ in~\cite{dereich2011euler} (see Remark~\ref{oss:prop:EB_IS_Y}), by proving that  the strong approximation error for $\hat{X}(t)$ in the time grid in $O(\sqrt{\Delta})$ (see Proposition~\ref{prop:EB_IS_Y}), and  in the whole interval~$[t_0,T]$ is $O(\sqrt{\Delta  |\log(\Delta  )|})$ (see  Theorem~\ref{thm:Teorema1}), under suitable conditions on $\sup_{t_k \in[t_0-\tau,t_0]}|X_0(t_k)-x_k|$ (see~\eqref{eq:HPappErr-X-0}),
and under the condition
\[
 \sigma^2< \frac{2}{1+ \left\lceil \frac{T-t_0}{\tau}\right\rceil}\,2 a \underline{\gamma},\quad \text{where } \underline{\gamma}=\inf_{t\in [t_0,T]} \gamma(t)>0,\text{ and   $\left\lceil x\right\rceil$ denotes the smallest integer $\geq x$.}
\]
Note that the above condition is stronger than the natural condition which guarantees that  the process $X^{(b)}(t)$ and the approximation scheme $y_k$   are both positive in the interval~$[t_0,T]$.\\

When $\tau$ is small, and smaller than the discretization step $\Delta$, one could use a different approach: First of all,
we observe that the fixed delay CIR model $X^{(b)}(t)$ is near the solution $V^{(b)}(t)$ of the equation
\begin{align}
dV^{(b)}(t)
&= (a-b) \left[\frac{a}{a-b}\,\gamma(t) -   V^{(b)}(t) \right]dt + \sigma \sqrt{V^{(b)}(t)} dW(t), \quad V^{(b)}(t_0)=X_0(t_0),\label{eq:CIR_a-b}
\end{align}
obtained by setting $\tau=0$  in~\eqref{eq:CIRdelayX}.
Indeed, by adding and subtracting the term  $bX^{(b)}(t)$ in~\eqref{eq:CIRdelayX}, we get
\begin{align*}
X^{(b)}(t)&=X_0(t) \qquad t_0-\tau\leq t \leq t_0,\\
dX^{(b)}(t)&=[a\gamma(t)-(a-b)X^{(b)}(t))+bX^{(b)}(t-\tau)-bX^{(b)}(t)]dt+\sigma\sqrt{X^{(b)}(t)}dW(t),\quad t>t_0.
\end{align*}
Hence, when $\tau$ is small, the difference $X^{(b)}(t-\tau)- X^{(b)}(t)$ is small (see Proposition~\ref{prop:p-thmoment}), and consequently, one could   approximate~$X^{(b)}(t)$ by an approximation of the solution of Eq.~\eqref{eq:CIR_a-b}. When $b<a$, Eq.~\eqref{eq:CIR_a-b} is a CIR model with deterministic long term depending on time, and then one can use the (suitably modified) approximation result by~\cite{dereich2011euler}, under the strong Feller condition  $\sigma^2<2a\underline{\gamma}$.
\\

The paper is organized as follows: Section~\ref{sec:SPR} has the aim to give some preliminary results on the moments of the processes $X(t)=X^{(b)}(t)$ and $Y(t)=Y^{(b)}(t)$ (for the sake of simplicity, we will write $X(t)$ and $Y(t)$ instead of $X^{(b)}(t)$ and $Y^{(b)}(t)$, unless necessary).   Section~\ref{sec:MB_IS_Y} and  Section~\ref{sec:EB_IS_Y} are devoted to  moment bounds and preliminary error bounds for the implicit Euler scheme,  extending  the corresponding results in~\cite{dereich2011euler}  to the fixed delay CIR model.   Section~\ref{sec:EBDISRES} is devoted to our main convergence result (Theorem~\ref{thm:Teorema1}).
 The paper ends with an appendix
   containing some nontrivial results on the $p$-moments of the classical CIR model: in particular we prove Lemma~\ref{lem:exBD},  a generalization of Lemma~$A.1$  in~\cite{bossy2007efficient}.

\section{Some Preliminary Results}\label{sec:SPR}
Let $(\Omega,\mathcal{F}, \mathbb{P})$ be a complete probability space with a right continuous filtration $\{\mathcal{F}_t\}_{t\geq t_0}$ and~$\mathcal{F}_{t_0}$ contains all $\mathbb{P}$-null sets. \\
The following standing assumptions hold:
\begin{assumptions}\label{ass:hp_EA}${}$
\begin{description}
\item[(i)] The process $W(t)$, $t\geq t_0$, is a Brownian motion with respect to the filtration $\mathcal{F}_t$, with $W(t_0)=0$, so that
    $\mathcal{F}_{t_0}$ is independent of  natural filtration $\mathcal{F}^{W}_t$.
\item[(ii)] The parameters $a$ and $\sigma$  are positive constants, and the parameter $b$ is a nonnegative constant.
\item[(iii)] The segment process $X_0(\cdot)$ is a positive continuous random function on $[t_0-\tau, t_0]$
    such that
    \begin{equation*}
    \label{eq:MeasurableCond-CIR}
    \int_{t_0-\tau}^{t_0}X_0(u) du < +\infty,\;\;\mathbb{P}\text{-a.s.};\end{equation*}
    moreover, we require that $X_0(t)$ is  measurable with respect to $\mathcal{F}_{t_0}$, for $t_0-\tau\leq t\leq t_0$, and therefore~$\sigma\{X_0(u)\,; u\in[t_0-\tau,t_0]\}$ is independent of $\mathcal{F}^{W}_t$, $t\geq t_0$.
\item[(iv)] The deterministic function $\gamma(t)$ is measurable, positive, and bounded on every bounded interval; in particular, in the time interval $[t_0,T]$,
    \begin{equation}\label{eq:BoundsForGamma}
    0<\underline{\gamma}:=\inf_{t\in[t_0,T]}\gamma(t)\leq \gamma(t)\leq\sup_{t\in[t_0,T]}\gamma(t):=\overline{\gamma}.\end{equation}
\end{description}
\end{assumptions}

We recall the following results without proofs. The interested reader is referred to~\cite{FloreNappo_Delay} for the proofs, which are based on the general results
of~\cite{DDLongTerm,deelstra_delbaen}.

\begin{proposition}\label{prop:FC_DELAY_EA}${}$\\
Under the Assumptions \ref{ass:hp_EA}, Eq.~\eqref{eq:CIRdelayX} with initial segment process~\eqref{eq:CIRdelayX-InitialSegment}  admits a unique solution $X(t)$.  Moreover, if the following inequality holds
\begin{equation}\label{eq:feller's_condition_delay_X}
\sigma^2\leq2a\gamma(t)\quad\text{for all $t\geq t_0$},
\end{equation}
then the process $X(t)$ is positive.
\end{proposition}

In other words condition~\eqref{eq:feller's_condition_delay_X} implies that the origin is unattainable.
\begin{proposition}\label{prop:integrabilityEA}${}$\\
 Under Assumptions~\ref{ass:hp_EA}, let the process $X(t)$ be the solution of Eq.~\eqref{eq:CIRdelayX} with initial segment process~\eqref{eq:CIRdelayX-InitialSegment}.\\
If furthermore
\begin{equation*}\label{eq:IntCondEA_HP}
\int_{t_0-\tau}^{t_0}\mathbb{E}\left[X_0(u)\right]du<+\infty, \quad \text{ and } \quad
\mathbb{E}\left[X_0(t_0)\right]<+\infty,\end{equation*}
then
\begin{enumerate}
\item for all $t\geq  t_0$, \quad $\mathbb{E}\left[\sup_{t_0\leq u \leq t}X(u)\right]<\infty$,
\item  for all $t\geq t^\prime\geq t_0$
       \begin{equation*}\label{eq:ExpValDelay_EA}
       \mathbb{E}\left[X(t)\right]=\mathit{e}^{-a(t-t^\prime)}\mathbb{E}\left[X(t^\prime)\right]
       +\int_{t^\prime}^t\mathit{e}^{-a(t-u)}\left(a\gamma(u)+b\mathbb{E}\left[X(u-\tau)
       \right]\right)du.\end{equation*}
\end{enumerate}
\end{proposition}
In the next proposition, we   prove that the fixed delay CIR process is larger than a classical CIR process; as a consequence, when the strong Feller condition holds, the negative moments are finite for all~$q>0$.
\begin{proposition}\label{prop:confronto}${}$\\
Under Assumptions~\ref{ass:hp_EA}, let $X(t)$ and $Y(t)$ be the solutions of Eq.~\eqref{eq:CIRdelayX} with initial segment process~\eqref{eq:CIRdelayX-InitialSegment}  and Eq.~\eqref{eq:diffProcessY} with initial segment process~\eqref{eq:diffProcessY-InitialSegment}, respectively, and let $\underline{X}(t)$ be the solution of the following classical CIR model
\begin{equation*}\label{eq:CIRX}
\begin{cases}
d\underline{X}(t)=&a\left(\underline{\gamma}-\underline{X}(t)\right)dt+\sigma\sqrt{\underline{X}(t)}dW(t),\\
\underline{X}(t_0)=&
\underline{X}_0,\end{cases}
\end{equation*}
where $\underline{\gamma}=\inf_{t\in[t_0,T]}\gamma(t)$.
Setting $\underline{Y}(t)=\sqrt{\underline{X}(t)}$, if $\underline{X}_0=X_0(t_0)$,  then
\begin{equation}\label{eq:comparison}
\mathbb{P}\left( \frac{1}{X(t)}\leq \frac{1}{\underline{X}(t)},\; t\in [t_0,T]\right)=\mathbb{P}\left( \frac{1}{Y(t)}\leq \frac{1}{\underline{Y}(t)},\; t\in [t_0,T]\right)=1.
\end{equation}
Moreover, assume that the strong Feller condition $\sigma^2 < 2 a \underline{\gamma}$ is satisfied,
and that
\begin{equation*}\label{eq:CONDuX0}
\mathbb{E}\left[\frac{1}{X^{\frac{1}{2}\nu}_0(t_0)}\right]=\mathbb{E}\left[\frac{1}{\underline{X}^{\frac{1}{2}\nu}_0}\right]<+\infty, \qquad \text{where }\nu=\frac{2a\underline{\gamma}}{\sigma^2}-1>0,\\
\end{equation*}
then, for all $q>0$ there exists a constant $c_q$ such that
\begin{equation}\label{eq:media-int-1suXalla-p}
        \mathbb{E}\left[ \left( \int_{t_0}^T \frac{1}{X(u)} du\right)^q\right]\leq c_{q} \left( 1+ \mathbb{E}\left[ \frac{1}{X^{\frac{1}{2}\nu}_0(t_0)}\right] \right).
        \end{equation}
\end{proposition}
\begin{proof}${}$\\
First of all, observe that, since $b \, X(t)$ is nonnegative, the comparison Theorem $1.1$ in~\cite{ikeda_watanabe1976} in each time interval~$[t_0+k\tau, t_0+(k+1)\tau]$, implies
        \[X(t)\geq \underline{X}(t)\;\text{for all $t\geq t_0$, $\mathbb{P}$-a.s.},\Longleftrightarrow
         \frac{1}{X(t)}\leq\frac{1}{\underline{X}(t)}\;\text{for all $t\geq t_0$, $\mathbb{P}$-a.s.},\]
and \eqref{eq:comparison} follows, together with inequality~\eqref{eq:media-int-1suXalla-p}, the latter being an immediate consequence of  Lemma~\ref{prop:Lemma3.1DNS} in the appendix.\\
\end{proof}

In the next proposition, we show that if the segment process $X_0(t)$, $t\in [t_0-\tau, t_0]$, has finite $p^{th}$-moments, then the same holds for the process $X(t)$, $t\in [t_0,T]$.
\begin{proposition}\label{prop:p-thmoment}${}$\\
Under Assumptions~\ref{ass:hp_EA}, let the process $X(t)$ be the solution of Eq.~\eqref{eq:CIRdelayX} with initial segment process~\eqref{eq:CIRdelayX-InitialSegment}.\\
If furthermore, for some  $p\geq 1$,
\begin{equation}\label{eq:p-thmoment}
\sup_{u\in[t_0-\tau,t_0]}\mathbb{E}\left[X^p_0(u)\right]du\leq K_p,\quad\text{where $K_p$ is a positive constant,}
\end{equation}
 then
\begin{enumerate}
\item the process  $X(t)$ has $p^{th}$-moments finite and uniformly bounded on bounded intervals,\\
\item for any $T\geq t_0$, there exists a constant $c_{1,p}$ such that
\begin{equation*}\label{eq:mediadifferenzaXhp}
\left(\mathbb{E}\left[|X(t)-X(s)|^p\right]\right)^\frac{1}{p}\leq c_{1,p}|t-s|^{\frac{1}{2}},\quad\text{for $t,s\in[ t_0, T]$}.
\end{equation*}
\end{enumerate}
Assume moreover that
\begin{equation}\label{eq:IntCondEA_HP2-unif}
\textstyle{\mathbb{E}\left[\sup_{t\in[t_0-\tau,t_0]}X_0^p(t) \right]\leq \mathcal{K}_p},
\end{equation}
then,
\begin{enumerate} \setcounter{enumi}{2}
\item for any $T\geq t_0$, there exists a constant $c_{2,p}$ such that
 \begin{equation}\label{eq:IntCondEA_HP2-unif-t_0T}
\textstyle{\mathbb{E}\left[\sup_{t\in[t_0,T]}X^p(t)\right]\leq c_{2,p}},
\end{equation}
\item for any $T\geq t_0$, there exists a constant $c_{3,p}$  such that
\begin{equation}\label{eq:MODULUS-IntCondEA_HP2-unif}
\big(\mathbb{E}\left[w_X^p(\delta;[t_0,T])\right]\big)^{\frac{1}{p}}\leq c_{3,p} \,(\delta |\log (\delta)|)^\frac{1}{2},
\end{equation}
where
$$
w_X(\delta;[t_0,T])=\sup_{s,t\in [t_0,T], |t-s|\leq\delta}|X(t)-X(s)|, \qquad \delta >0,
$$
is the modulus of continuity of the process $X(t)$.
\end{enumerate}
\end{proposition}
\begin{proof}${}$\\
For the first two points, the idea is to show that the statement holds true on the interval $[t_0,t_0+\tau]$ and to repeat the procedure by induction on the intervals $[t_0+(k-1)\tau,t_0+k\tau]$ with $k\leq m$ and $m$ chosen such that $t_0+(m-1)\tau<T\leq t_0+m\tau$.\\
\begin{enumerate}
\item[\emph{1.}]  On the interval $[t_0,t_0+\tau]$, the unique solution satisfies
\begin{equation*}\label{eq:XsolCIR}
\begin{split}X(t)=&\mathit{e}^{-a(t-t_0)}\left(X(t_0)+\int_ {t_0}^t(a\gamma(u)+bX_0(u-\tau))\mathit{e}^{a(u-t_0)}du
+\sigma\int_{t_0}^t
\mathit{e}^{a(u-t_0)}\sqrt{X(u)}\, dW(u)\right).\end{split}\end{equation*}
Consequently,
\begin{align*}
X^p(t)
\leq&3^{p-1}\mathit{e}^{-pa(t-t_0)}\left[X^p(t_0)+\left|\int_ {t_0}^t(a\gamma(u)+bX_0(u-\tau))\mathit{e}^{a(u-t_0)}du\right|^p
\right. \\
&\phantom{xxxxxxxxxxx}\left.
+\left|\sigma\int_{t_0}^t
\mathit{e}^{a(u-t_0)}\sqrt{X(u)}dW(u)\right|^p\right]. 
\end{align*}
We define the stopping time $\tau_R=\inf(t\geq t_0\,:\,X(t)\geq R)$, and have that for all $t\leq t_0+\tau$
\begin{equation}\label{eq:mediaX2}\begin{split}
\mathbb{E}\left[X^p(t\wedge\tau_R)\right]\leq&3^{p-1}\mathit{e}^{-pa (t-t_0)}\left\{\mathbb{E}\left[X^p(t_0)\right]+\mathbb{E}\left[\left|\int_ {t_0}^{t\wedge\tau_R}(a\gamma(u)+bX_0(u-\tau))\mathit{e}^{a(u-t_0)}du\right|^p\right]\right.\\
&\phantom{xxxxxxxxxxxxxx}\left.+\mathbb{E}\left[\left|\sigma\int_{t_0}^{t\wedge\tau_R}
\mathit{e}^{a(u-t_0)}\sqrt{X(u)}dW(u)\right|^p\right]\right\}.
\end{split}\end{equation}
 Burkholder-Davis-Gundy inequality (see, e.g.,~\cite{RevuzYor}) implies that
\begin{align*}
\mathit{e}^{-p
a(t-t_0)}\mathbb{E}\left[\left|\sigma\int_{t_0}^{t\wedge\tau_R}
\mathit{e}^{a(u-t_0)}\sqrt{X(u)}dW(u)\right|^p\right]
&=\mathbb{E}\left[\left|\sigma\int_{t_0}^{t}
\mathit{e}^{-a(t-u)}\sqrt{X(u)}\mathbf{1}_{\{u\leq\tau_R\}}dW(u)\right|^p\right]
\\
&\leq  C_p \sigma^p\mathbb{E}\left[\left(\int_{t_0}^{t}
\mathit{e}^{-2a(t-u)}X(u)\mathbf{1}_{\{u\leq\tau_R\}}\,du\right)^{\frac{p}{2}}\right]
\\
&\leq C_p\sigma^p\left(\mathbb{E}\left[1+\left(\int_{t_0}^{t}  \mathit{e}^{-2a(t-u)}X(u)
\mathbf{1}_{\{u\leq\tau_R\}} \,du\right)^{p}\right]\right)
\\
&\leq  C_p \sigma^p\left(1+\tau^{p-1}\,
\int_{t_0}^t\sup_{u^\prime\in[t_0,t_0+u]}
\mathbb{E}\left[X^p(u^\prime\wedge\tau_R)\right]du\right),
\end{align*}
where $C_p$ is a universal constant.

Condition~\eqref{eq:p-thmoment} and H\"{o}lder   inequality imply that
\begin{align*}
&\mathit{e}^{-p a(t-t_0)}
\mathbb{E}\left[\left(\int_
{t_0}^{t\wedge\tau_R}(a\gamma(u)+bX_0(u-\tau))\mathit{e}^{a(u-t_0)}du\right)^p\right]
\\
\leq&2^{p-1}\left(\int_ {t_0}^{t\wedge\tau_R}a\gamma(u)du\right)^p
+2^{p-1}\mathbb{E}\left[\left(\int_ {t_0}^{t\wedge\tau_R}bX_0(u-\tau)du\right)^p\right]\\
\leq&2^{p-1}a^p\overline{\gamma}^p\tau^p+2^{p-1}b^pK_p\tau^p,
\end{align*}
where $\overline{\gamma}$ is defined by~\eqref{eq:BoundsForGamma}.\\
Consequently \eqref{eq:mediaX2} is upper bounded by
\begin{equation*}\label{eq:mediaX2bis}\begin{split}
\mathbb{E}\left[X^p(t\wedge\tau_R)\right]\leq& 3^{p-1}\mathit{e}^{-pa(t-t_0)}\,\mathbb{E}\left[X^p(t_0)\right]
+6^{p-1}a^p\overline{\gamma}^p\tau^p+6^{p-1}\,b^p\,K_p\,\tau^p\\
&+3^{p-1}  C_p \sigma^p\mathit{e}^{pa\tau}\left(1+\tau^{p-1}\,\int_{t_0}^{t}\sup_{u^\prime\in[t_0,t_0+u]}
\mathbb{E}\left[X^p(u^\prime \wedge \tau_R)\right]du\right).
\end{split}\end{equation*}
By Gronwall inequality and~\eqref{eq:p-thmoment}, letting $R\rightarrow +\infty$ (and hence for $\tau_R\rightarrow +\infty$), we get
\begin{equation*}
\sup_{u\in[t_0,t_0+\tau]}\mathbb{E}\left[X^p(u)\right]du\leq K_1,\quad\text{where $K_1$ is a positive constant.}
\end{equation*}
Repeating this procedure by induction on the intervals $[t_0+(k-1)\tau,t_0+k\tau]$ with $k=2,3...$, we have that the process $X^p(t)$ is integrable for all $t\in[t_0,+\infty)$, with $p$-moments uniformly bounded on bounded intervals.
\item[\emph{2.}]
For $s,\,t$ in the interval $[t_0,T]$, we have that
\[X(t)-X(s)=a\int_{s}^t(\gamma(u)-X(u))du+b\int_{s}^t X(u-\tau)du+\sigma\int_{s}^t\sqrt{X(u)}dW(u),\]
and
\begin{align*}
\mathbb{E}\left[|X(t)-X(s)|^p\right]^{\frac{1}{p}}\leq&
\mathbb{E}\left[\left|a\int_{s}^t(\gamma(u)-X(u))du\right|^p\right]^\frac{1}{p}+
\mathbb{E}\left[\left|b\int_{s}^tX(u-\tau)du\right|^p\right]^\frac{1}{p}\\
&+\mathbb{E}\left[\left|\sigma\int_{s}^t\sqrt{X(u)}dW(u)\right|^p\right]^\frac{1}{p}.
\end{align*}
By H\"{o}lder inequality, the sum of the first two addends is bounded above by
\begin{align*}
&2^{\frac{p-1}{p}}a|t-s|  \left(\sup_{t_0\leq u\leq T}\gamma^p(u)+\sup_{t_0\leq u\leq T}\mathbb{E}\left[|X(u)|^p\right]\right)^\frac{1}{p}
+b|t-s| \left(\sup_{t_0-\tau\leq u\leq T-\tau}\mathbb{E}\left[|X(u)|^p\right]\right)^\frac{1}{p}.
\end{align*}
An upper bound for the last term of the previous inequality, is obtained using Burkholder-Davis-Gundy inequality:
\begin{align*}
\mathbb{E}\left[\left|\int_{s}^t\sqrt{X(u)}dW(u)\right|^p\right]^\frac{1}{p}
\leq&C_p|t-s|^{\frac{1}{2}}\left(\sup_{t_0\leq u\leq t_0+\tau} \mathbb{E}\left[|X(u)|^{\frac{p}{2}}\right]\right)^\frac{1}{p}.
\end{align*}
By part $1$, we get the result.
\item[\emph{3.}] First of all observe that
\begin{align*}
\sup_{s\in [t_0, t]} X^p(s\wedge \tau_R)
\leq&3^{p-1} \left[X^p(t_0)+\left(\int_ {t_0}^t\sup_{u^\prime\in[t_0,u]}\big(a\gamma(u^\prime)+ aX(u^\prime)+bX(u^\prime-\tau)\big) I_{[t_0,\tau_R]}(u)
du\right)^p\right.  \\
&\phantom{xxxxxxx}\left.+\sup_{s\in [t_0, t]}\left|\sigma\int_{t_0}^s I_{[t_0,\tau_R]}(u)
\sqrt{X(u)}dW(u)\right|^p\right].
\end{align*}
Then, since the function $\gamma(t)$ is upper-bounded by $\overline{\gamma}$ on $[t_0, T]$ (see \eqref{eq:BoundsForGamma}),
 by   taking the expectations and, similarly to the proof of point~\emph{1.}, by H\"{o}lder inequality and Burkholder-Davis-Gundy inequality,
 we get that
 \begin{align*}
&\mathbb{E}\big[\sup_{s\in [t_0, t]} X^p(s\wedge \tau_R)\big]
\leq 3^{p-1}  \mathbb{E}\big[X^p(t_0)\big]\\
 &+3^{2(p-1)}(T-t_0)^{p-1} \, \int_ {t_0}^t \big(a^p\overline{\gamma}^p +
 b^p\mathbb{E}\big[\sup_{v^{\prime}\in[t_0-\tau,t_0]} X^p(v^{\prime})\big]+(a+b)^p \mathbb{E}\big[\sup_{u^\prime\in[t_0,u]} X^p(u^\prime\wedge \tau_R)\big] \big) du
\\ &  +  3^{p-1} \sigma^{2p}\left[1+ (T-t_0)^{p-1} \, \int_{t_0}^t
\sup_{u^\prime\in[t_0,u]} \mathbb{E}\big[X^p(u^\prime\wedge \tau_R)\big]du \right].
\end{align*}
Since condition~\eqref{eq:IntCondEA_HP2-unif} implies condition~\eqref{eq:p-thmoment}, we can use point~\emph{1.}, and then by Gronwall inequality, and letting $R$ go to infinity, we get the result, i.e., \eqref{eq:IntCondEA_HP2-unif-t_0T}.
\item[\emph{4.}]
 We can apply Theorem~1  in~\cite{fischer2009moments} and get the bounds~\eqref{eq:MODULUS-IntCondEA_HP2-unif} for the modulus of continuity if we find two random variables  $\zeta$ and $\xi$,   with $\mathbb{E}\left[\zeta^p\right]< \infty$, $\mathbb{E}\left[\xi^{\frac{p}{2}+\varepsilon}\right] < \infty$, for some $\varepsilon >0$, and such that  for any $s,\,t\in [t_0,T]$
$$
\left|
\int_s^t \big|a(\gamma(u) -X(u))+ bX(u-\tau)\big| du \right|\leq\zeta \,|t-s|,
 $$
and
$$
\left|
\int_s^t \sigma^2 X(u) du \right|\leq \xi \,|t-s|.
$$
We can take
$$
\zeta:= a \overline \gamma + a \sup_{u\in[t_0,T]} X(u) + b \sup_{u\in[t_0-\tau,T-\tau]}X(u),\qquad \xi:= \sigma^2 \sup_{u\in[t_0,T]} X(u),
$$
and observe that,  by condition~\eqref{eq:IntCondEA_HP2-unif} and the previous point \emph{3.},
the random variables $\zeta$ and $\xi$  have finite $p$-moments:
$$
\sup_{u\in[t_0-\tau,T]} X^p(u) \leq \max\left( \sup_{u\in[t_0,T]} X^p(u),  \sup_{u\in[t_0-\tau,t_0]} X^p(u)\right).
$$
\end{enumerate}
\end{proof}
As a straightforward consequence of the previous proposition, we now extend the preliminary results of~\cite{dereich2011euler} (see Lemma~3.1 and Lemma~3.2 therein) to our model.
\begin{corollary}\label{cor:meanY}${}$\\
Assume the same conditions of Proposition~\ref{prop:p-thmoment}. If condition~\eqref{eq:p-thmoment} holds, then
\begin{enumerate}
\item the process  $Y(t)=\sqrt{X(t)}$ has $2p^{th}$-moments finite and uniformly bounded on bounded intervals,\\
  \item \begin{equation*}\label{eq:mediaDifferenzagen}
         \left(\mathbb{E}\left[|Y(t)-Y(s)|^{2p}\right]\right)^\frac{1}{p}\leq c_{1,p}|t-s|^\frac{1}{2},\quad\text{for $t$, $s$ $\in[t_0,T]$}.
         \end{equation*}
  \end{enumerate}
  If moreover~\eqref{eq:IntCondEA_HP2-unif} holds, then
  \begin{enumerate} \setcounter{enumi}{2}
  \item \begin{equation*}\label{eq:mediaYgen}
                       \textstyle{\mathbb{E}\left[\sup_{u\in[t_0-\tau,T]}|Y(u)|^{2p}\right]\leq c_{2,p}<+\infty},
        \end{equation*}
  \item \begin{equation*}\label{eq:mediaSupDifferenzagen}
        \mathbb{E}\left[w_Y^{2p}(\delta;[t_0,T])\right]
         \leq c_{3,p}\left(\left|\log\left(\delta\right)
         \right|\delta\right)^p,
         \end{equation*}
\end{enumerate}
where $w_Y(\delta;[t_0,T])$ is the modulus of continuity of the process $Y(t)$, and the constants  $c_{i,p}$, $i=1,2,3$, are defined in Proposition~\ref{prop:p-thmoment}.
\end{corollary}
\begin{proof}${}$\\
The proof of Corollary~\ref{cor:meanY} follows immediately from points \emph{1., 2., 3.} and \emph{4.} of Proposition~\ref{prop:p-thmoment}, thanks to the inequality $|\sqrt{x}-\sqrt{y}|\leq\sqrt{|x-y|}$.
\end{proof}

\section{Moment Bounds for the Euler Scheme for the process $Y(t)$}\label{sec:MB_IS_Y}
In this section, we deal with the approximation scheme $y_k$, $k\geq 0$ defined in~\eqref{eq:ApproximationScheme}. Following~\cite{dereich2011euler}, our aim is to show that the approximation scheme $y_k$, $k\geq 0$ has second moments uniformly bounded.
 We will use the notations~\eqref{eq:notations} together with
\begin{equation}\label{eq:a-star}
\underline{a}^* =\sup_{t\in[t_0,T]}\underline{a}(t).
\end{equation}
 We recall that the discretization step
$\Delta=\Delta t_k = t_{k+1}-t_{k}=\frac{\tau}{N}$,
 so that the delay time $\tau$ is proportional to $\Delta $, and we can also consider instead of $[t_0,T]$ the time interval $[t_0, t_0+m\tau]$  where $m=\lceil \frac{T-t_0}{\tau}\rceil$, i.e., $m$  is such that $t_0+(m-1)\tau<T\leq t_0+m\tau$.
\begin{lemma}\label{lem:MBIS}${}$\\
If the following condition holds true
\begin{equation*}\label{eq:MBIShp}
\sup_{k\,:\,t_k\in[t_0-\tau,t_0]}\mathbb{E}\left[y^2_{k}\right]
\leq K_0,
\end{equation*}
then, the second moment of the approximation scheme $y_k$ are uniformly bounded on bounded intervals, i.e., for any $T>t_0$, there exists a constant $K_T$ such that
\begin{equation*}\label{eq:MBISth}
\sup_{k\,:\,t_k\leq T} \mathbb{E}[y^2_k] \leq K_T.
\end{equation*}
\end{lemma}
\begin{proof}${}$\\
For the ease of the reader, we recall the approximation scheme~\eqref{eq:ApproximationScheme}:
\[
y_{k+1}=y_k+\left(\underline{a}(t_{k+1})\frac{1}{y_{k+1}} - \overline{a} y_{k+1}\right)\Delta t_k+\overline{b} \,\frac{y^2_{k+1-N}}{y_{k+1}}\, \Delta t_k+   \overline{\sigma} \,\Delta W_k,\quad\text{with } \Delta W_k =W(t_{k+1})-W(t_k).
\]
Multiplying both sides by $y_{k+1}$, we obtain
\[y^2_{k+1}=\left(\underline{a}(t_{k+1})+\overline{b}\,y^2_{k+1-N}-\overline{a}\,y^2_{k+1}\right)\Delta t_k+
\left(\overline{\sigma}\,\Delta W_k+y_k\right)y_{k+1},
\]
then, taking into account that $\overline{a}>0$, and that
\[
\left(\overline{\sigma}\,\Delta W_k+y_k\right)y_{k+1}\leq
\frac{1}{2}\left(\overline{\sigma}\,\Delta W_k+y_k\right)^2+\frac{1}{2}y^2_{k+1},
\]
we obtain
\begin{align*}
\frac{y^2_{k+1}}{2}\leq&\left(\underline{a}^*+\overline{b}\,y^2_{k+1-N}\right)\Delta t_k+
\frac{1}{2}\left(\overline{\sigma}\,\Delta W_k+y_k\right)^2,\end{align*}
where $\underline{a}^*$ is defined in~\eqref{eq:a-star}.\\
Adding ad subtracting $\frac{\overline{\sigma}^2}{2}\Delta t_k$ and
multiplying both sides by $2$, we have
\[
y^2_{k+1}\leq\left(2\underline{a}^*+\overline{\sigma}^2\right)\Delta t_k+2\overline{b}y_{k+1-N}^2\Delta t_k
+y_k^2+M_k,
\]
where
$M_k=2\overline{\sigma}y_k\Delta W_k+\overline{\sigma}^2\,\left((\Delta W_k)^2-\Delta t_k\right),$
is a discrete time martingale difference. Consequently, we have that
\begin{equation}\label{eq:MB1}\begin{split}
y^2_{k}&=\sum_{j=0}^{k-1}(y^2_{j+1}-y^2_j)+y^2_0\leq y_0^2+\left(2\underline{a}^*+\overline{\sigma}^2\right) (t_{k}-t_0)+2\overline{b}\sum_{j=1-N}^{k-N}y_j^2\Delta t_j+
\sum_{j=0}^{k-1} M_j,\end{split}\end{equation}
and
\begin{equation*}\label{eq:MB2}
\mathbb{E}\left[y^2_{k}\right]\leq \mathbb{E}\left[y_0^2\right]+\left(2\underline{a}^*+\overline{\sigma}^2\right) (t_{k }-t_0)
+2\overline{b}\sum_{j=1-N}^{k -N}\mathbb{E}\left[y_j^2\right]\Delta t_j.\end{equation*}
In the interval $[t_0,t_0+\tau]$, we have
\[\sup_{k\,:\,t_k\in[t_0,t_0+\tau]}\mathbb{E}\left[y^2_{k}\right]\leq \mathbb{E}\left[y_0^2\right]+\left(2\underline{a}^*+\overline{\sigma}^2+2\overline{b}K_0\right)\tau.\]
By induction, we have the statement.
\end{proof}

Thanks to the following result, we determine moment bounds for the implicit Euler scheme for $Y(t)=Y^{(b)}(t)$.
\begin{proposition}\label{prop:DNS3.4}${}$\\
If the following condition holds
\begin{equation}\label{eq:MBIS_prophp}
\sup_{k\,:\,t_k\in[t_0-\tau,t_0]}\mathbb{E}\left[|y_k|^{2p}\right]\leq K_{2p},
\end{equation}
for all $p\geq 1$, then, for any $T>t_0$, there exists a constant $K_{2p,T}$ such that
\begin{equation}\label{eq:MBIS_PROP}
\mathbb{E}\left[\sup_{k:\, t_0\leq t_k\leq T} |y_k|^{2p}\right]\leq K_{2p,T}.
\end{equation}
\end{proposition}
\begin{proof}${}$\\
The idea is to start with the first interval $[t_0,t_0+\tau]$, and  show that
\begin{equation}\label{eq:MBIS_PROP-J-1}
\mathbb{E}\left[\sup_{k:\, t_k\in[t_0,t_0+\tau]} |y_k|^{2p}\right]<+\infty,\quad \text{for $p=2^\ell$},
\end{equation}
by using induction on $\ell$.\\

From \eqref{eq:MB1}, for $k\geq 1$, we obtain that
\begin{align*}
y_k^2\leq&y_0^2+
\left(2\underline{a}^*+\overline{\sigma}^2\right)\sum_{j=1}^k\Delta t_{j-1}+2\overline{b}\sum_{j=1}^ky_{j-N}^2\Delta t_{j-1}+2\overline{\sigma}\sum_{j=1}^ky_{j-1}\Delta W_{j-1}\\&+\overline{\sigma}^2\,\sum_{j=1}^k\left((\Delta W_{j-1})^2-\Delta t_{j-1}\right).
\end{align*}
Recalling that
\[(\Delta W_j)^2-\Delta t_j=2\int_{t_j}^{t_{j+1}}W(s)dW(s),\]
we obtain
\begin{equation*}\label{eq:MB3}\begin{split}
\sup_{k\,:\,t_k \in[t_0,t_0+\tau]} |y_k|^2\leq
&c_1+c_2\frac{\tau}{N}\sum_{j=1}^N y_{j-N}^2+c_3\sup_{k\,:\,t_k \in[t_0,t_0+\tau]}\left|\int_{t_0}^{t_k}m_t dW_t\right|,\\
\end{split}\end{equation*}
where
\begin{equation}\label{eq:MB4}m_t=2\overline{\sigma}y_k+2\overline{\sigma}^2W_t,\quad t\in[t_k,t_{k+1}).\end{equation}
Raising to the  $p$-power,  we get
\[
\sup_{k\,:\,t_k \in[t_0,t_0+\tau]} |y_k|^{2p}\leq 3^{p-1}\left(c^p_1+c^p_2\left(\frac{\tau}{N}\sum_{j=1}^N y_{j-N}^2\right)^p+c^p_3\sup_{k\,:\,t_k \in[t_0,t_0+\tau]}\left|\int_{t_0}^{t_k}m_t dW_t\right|^p\right).
\]
Consequently, we have
\begin{equation}\label{eq:MB5}\begin{split}
\mathbb{E}\left[\sup_{k\,:\,t_k \in[t_0,t_0+\tau]}|y_k|^{2p} \right]\leq& 3^{p-1}\left(c^p_1+c^p_2\mathbb{E}\left[\left(\frac{\tau}{N}\sum_{j=1}^N y_{j-N}^2\right)^p\right]\right.\\
&\left.+c^p_3\mathbb{E}\left[\sup_{k\,:\,t_k \in[t_0,t_0+\tau]}\left|\int_{t_0}^{t_k}m_t dW_t\right|^p\right]\right).\end{split}
\end{equation}

 By H\"{o}lder inequality applied to the measure $\frac{1}{N}\sum_{j=1}^N \delta_j(dx)$ and by condition~\eqref{eq:MBIS_prophp}, we obtain
\begin{align}\label{eq:MB5-BIS}
\mathbb{E}\left[\left(\frac{\tau}{N}\sum_{j=1}^N y_{j-N}^2\right)^p\right]&
\leq\tau^p \mathbb{E}\left[\frac{1}{N}\sum_{j=1}^N y_{j-N}^{2p}\right]
\leq
\tau^p\sup_{k\,:\,t_k\in[t_0-\tau,t_0]}\mathbb{E}\left[|y_k|^{2p}\right]<\infty.
\end{align}

By Burkholder-Davis-Gundy inequality  and by \eqref{eq:MB4}, we get
\begin{align}\notag
\mathbb{E}\left[\sup_{k\,:\,t_k \in[t_0,t_0+\tau]}\left|\int_{t_0}^{t_k}m_t dW_t\right|^p\right]&\leq
C_p\mathbb{E}\left[\left(\int_{t_0}^{t_0+\tau}|m_t|^2 dt\right)^{\frac{p}{2}}\right]\\
\notag
&\leq C_p\,\tau^{\frac{p}{2}}\sup_{t\in[t_0,t_0+\tau]}\mathbb{E}\left[|m_t|^p\right]\\
&\leq C_p\,2^{2p-1}\,\overline{\sigma}^{p}\,\tau^{\frac{p}{2}}\left(\overline{\sigma}^{p}\,\tau+\sup_{k\,:\,t_k \in[t_0,t_0+\tau]}\mathbb{E}\left[|y_k|^p\right]\right).\label{eq:MB5-TER}
\end{align}

By~\eqref{eq:MB5}, \eqref{eq:MB5-BIS} and \eqref{eq:MB5-TER}, for some constant $c_i(p,\tau)$, we get
\begin{align}\label{eq:MB5-QUATER}
\mathbb{E}\left[\sup_{k\,:\,t_k\in[t_0,t_0+\tau]}|y_k|^{2p}\right]
&\leq c_1(p,\tau)+c_2(p,\tau)\sup_{k\,:\,t_k\in[t_0,t_0+\tau]}
\mathbb{E}\left[|y_k|^p\right]+c_3(p,\tau)\sup_{k\,:\,t_k\in[t_0-\tau,t_0]}\mathbb{E}\left[|y_k|^{2p}\right].
\end{align}

The case $p=1$ is then obvious. Taking $p=2$, we obtain
\[
\mathbb{E}\left[\sup_{k\,:\,t_k\in[t_0,t_0+\tau]}|y_k|^{4}\right]\leq c_1(2,\tau)
+c_2(2,\tau)\sup_{k\,:\,t_k\in[t_0,t_0+\tau]}
\mathbb{E}\left[|y_k|^2\right]+c_3(2,\tau)\sup_{k\,:\,t_k\in[t_0-\tau,t_0]}\mathbb{E}\left[|y_k|^{4}\right].
\]

Lemma~\ref{lem:MBIS} and assumption~\eqref{eq:MBIS_prophp} for $p=2$, imply
\[\mathbb{E}\left[\sup_{k\,:\,t_k\in[t_0,t_0+\tau]}|y_k|^{4}\right]<+\infty.\]
By induction on $\ell$, using~\eqref{eq:MB5-QUATER}, we have the statement~\eqref{eq:MBIS_PROP-J-1} in the first interval.\\

Finally, by induction on the intervals $[t_0+k\tau,t_0+(k+1)\tau]$ with $k=1,2,\ldots$, we get the thesis.\\
\end{proof}

\section{Error bound for the Implicit Euler Scheme}\label{sec:EB_IS_Y}
Let $f\,:\,\mathbb{R}_+\times\mathbb{R}_+\times\mathbb{R}_+\rightarrow\mathbb{R}_+$ be the functions defined as follows
\begin{equation}\label{eq:def_f(t,y,z)}
f(t, y, z )= \underline{a}(t) \frac{1}{y} -\overline{a}
y+\overline{b}\frac{z^2}{y},
\end{equation}
then, we can represent the approximation scheme as follows
\begin{equation}\label{eq:AS_f}
y_{k+1} =y_k+ f(t_{k+1}, y_{k+1}, y_{k+1-N}) \Delta t_k + \sigma \Delta
W_k.\end{equation}

\begin{lemma}\label{lem:DisugualianzePerf}${}$\\
For all $y, y^\prime, z, z^\prime >0$ and for all  $t,\,t^\prime\geq t_0$, the function $f$, defined above, satisfies the following inequalities:
\begin{enumerate}
  \item
  \begin{equation}\label{eq:FirstInequalityForf}
  (y-y^\prime)\,\left[f(t,y,z)-f(t,y^\prime,z^\prime)\right]\leq \overline{b} (y-y^\prime)\,\frac{1}{y^\prime}\, [ z^2  -  (z^\prime)^2  ];
  \end{equation}
  \item
  \begin{equation}\label{eq:SecondInequalityForf}
  |f(t,y,z)-f(t^\prime,y^\prime,z^\prime)|\leq| \underline{a}(t)- \underline{a}(t^\prime)| \, \frac{1}{y^\prime}
  + \left((\underline{a}(t)+\overline{b}  \,z^2 )\,\frac{1}{y y^\prime} +\overline{a} \right)
\left|  y   -  y^\prime \right| +
 \overline{b}  \,\frac{1}{y^\prime}\,  | z^2   -  (z^\prime)^2   |.
  \end{equation}
\end{enumerate}
\end{lemma}
\begin{proof}${}$\\
The results follow by simple computations:
\begin{enumerate}
  \item[\emph{1.}]
\begin{align*}
&(y-y^\prime)\,\left[f(t,y,z)-f(t,y^\prime,z^\prime)\right]=
(y-y^\prime)\,\left[ \underline{a}(t) \frac{1}{y} -\overline{a}
y+\overline{b}\frac{z^2}{y} - \big(\underline{a}(t) \frac{1}{y^\prime} -\overline{a}
y^\prime +\overline{b}\frac{(z^\prime)^2}{y^\prime} \big)\right]
\\=& (y-y^\prime)\,\left[ \underline{a}(t) \frac{1}{y}   -  \underline{a}(t) \frac{1}{y^\prime}
\right]- \overline{a}(y-y^\prime)\,\left[
y -
y^\prime  \right] + \overline{b} (y-y^\prime)\,\left[\frac{z^2}{y} \mp\frac{z^2}{y^\prime} - \frac{(z^\prime)^2}{y^\prime} \right]
\\ \leq & 0 - \overline{a}  (y-y^\prime)^2 + \overline{b} (y-y^\prime)\,\left[\frac{z^2}{y} - \frac{z^2}{y^\prime}\right]
+ \overline{b} (y-y^\prime)\,\left[\frac{z^2}{y^\prime} - \frac{(z^\prime)^2}{y^\prime} \right]
\\
\leq & 0 + \overline{b} (y-y^\prime)\,\frac{1}{y^\prime}\, [ z^2  -  (z^\prime)^2  ]
\end{align*}
  \item[\emph{2.}] \begin{align*}
&|f(t,y,z)-f(t^\prime,y^\prime,z^\prime)|
=\left| \underline{a}(t) \frac{1}{y} -\overline{a}
y+\overline{b}\frac{z^2}{y} - \big(\underline{a}(t^\prime) \frac{1}{y^\prime} -\overline{a}
y^\prime +\overline{b}\frac{(z^\prime)^2}{y^\prime} \big)\right|
\\=&
| \underline{a}(t)- \underline{a}(t^\prime)| \, \frac{1}{y^\prime}
  + \underline{a}(t)\,\left|  \frac{1}{y}   -  \frac{1}{y^\prime}
\right| +\overline{a}\,\left|
y -
y^\prime  \right| + \overline{b} \,\left|\frac{z^2}{y} \mp\frac{z^2}{y^\prime} - \frac{(z^\prime)^2}{y^\prime} \right|
\\\leq &
| \underline{a}(t)- \underline{a}(t^\prime)| \,  \frac{1}{y^\prime}
   +
\underline{a}(t)\,\frac{1}{y y^\prime}\,\left|  y   -  y^\prime
\right| +\overline{a}\,\left|
y -
y^\prime  \right| +   \overline{b}  \,z^2 \,\frac{1}{y y^\prime}\,\left|  y   -  y^\prime
\right|
+ \overline{b} \,\frac{1}{y^\prime} \,\left|z^2 - (z^\prime)^2 \right|
\\
=&
| \underline{a}(t)- \underline{a}(t^\prime)| \, \frac{1}{y^\prime}
 + \left((\underline{a}(t)+\overline{b}  \,z^2 )\,\frac{1}{y y^\prime} +\overline{a} \right)
\left|  y   -  y^\prime \right| +
 \overline{b}  \,\frac{1}{y^\prime}\,  | z^2   -  (z^\prime)^2   |.
\end{align*}
\end{enumerate}
\end{proof}

Now, we make the following further standing assumptions.
\begin{assumptions}\label{ass:InitialValue}${}$
\begin{description}
        \item [(i)] The process $X_0(t)$ is a Borel measurable for $t\in[t_0-\tau,t_0]$ such that for any $p>0$
\begin{equation}\label{eq:IntCondP2}
\textstyle{\mathbb{E}\left[\sup_{t\in[t_0-\tau,t_0]}|X_0(t)|^{p}\right]<+\infty}.
\end{equation}
        \item[(ii)] For any $0< q<\frac{2a\underline{\gamma}}{\sigma^2}$
  \begin{equation*}\label{eq:IntCondP3}
 \mathbb{E}\left[\frac{1}{X^q_0(t_0)}\right] <+\infty.\end{equation*}
 \item [(iii)] the function $\gamma(t)$ is H\"{o}lder continuous of order $\frac{1}{2}$ , i.e.,
\begin{equation}\label{eq:holdercond}
|\gamma(t)-\gamma(s)|\leq L|t-s|^{\frac{1}{2}}, \qquad t,s \in [t_0,T];
\end{equation}
  \item [(iv)] the parameters $a$, $\underline{\gamma}$ and $\sigma$ satisfy the following condition
         \begin{equation*}\label{eq:StrongParCond}
         \frac{2\,a\underline{\gamma}}{\sigma^2}>
         \frac{1+\left\lceil\tfrac{T-t_0}{\tau}\right\rceil}{2};
         \end{equation*}
\end{description}
\end{assumptions}

Now, we show that the numerical scheme converges on the discretization points. The proof is an extension of Proposition $3.3$ in~\cite{dereich2011euler}.

\begin{proposition}\label{prop:EB_IS_Y}${}$\\
Beside Assumptions~\ref{ass:hp_EA} and~\ref{ass:InitialValue}, assume $b>0$, condition~\eqref{eq:MBIS_prophp},  and
\begin{equation}\label{eq:HPappErr}
\left(\mathbb{E}\left[|E^Y_{(0)}|^p\right]\right)^{\frac{1}{p}}\leq C_{p,0}^Y\left(\tfrac{\tau}{N}\right)^{\frac{1}{2}}= C_{p,0}^Y \Delta ^{\frac{1}{2}},\quad\text{for all $p\geq 1$},
\end{equation}
where $E^Y_{(0)}:=\sup_{t_k\in [t_0-\tau, t_0]}|Y(t_k)-y_k|$
is the initial error.\\
Then, for any $p\in\left[1,\frac{2\,a\underline{\gamma}}{\sigma^2}\frac{2}{1+\left\lceil\frac{T-t_0}{\tau}\right\rceil}\right)$,
\begin{equation*}\label{eq:EB_IS_Y-p}
\left(\mathbb{E}\left[\textstyle{\sup_{k\,:\,t_k\in[t_0,T]}}|Y(t_k)-y_k|^p\right]\right)^{\frac{1}{p}} \leq
C^Y\left(\tfrac{\tau}{N}\right)^{\frac{1}{2}}=C^Y \Delta ^{\frac{1}{2}}.
\end{equation*}
and
\begin{equation*}\label{eq:EB_IS_X-p}
\left(\mathbb{E}\left[\textstyle{\sup_{k\,:\,t_k\in[t_0,T]}}|X(t_k)-x_k|^p\right]\right)^{\frac{1}{p}} \leq
C^X\left(\tfrac{\tau}{N}\right)^{\frac{1}{2}}=C^X \Delta ^{\frac{1}{2}}.
\end{equation*}
\end{proposition}
\begin{proof}${}$\\
First of all, since for every $p>1$
\[\textstyle{\mathbb{E}\left[\sup_{k\,:\,t_k\in[t_0,T]}|Y(t_k)-y_k|\right]} \leq\left(\textstyle{\mathbb{E}\left[\sup_{k\,:\,t_k\in[t_0,T]}|Y(t_k)-y_k|^p\right]}\right)^{\frac{1}{p}},\]
it is sufficient to prove the statement for $p>1$.\\
We introduce the following notations.
Let $e_k$ be the sequence of the approximation errors defined as follows
\begin{equation}\label{eq:ErrorSequence}
e_k:=Y(t_k)-y_k,
\end{equation}
and let $E^Y_{(h)}$ be the approximation error on the time interval $J_h=[t_0+(h-1)\tau,t_0+h\tau] $, that is
\begin{equation*}\label{eq:ErrorInterval}
E^Y_{(h)}:=\sup_{t_k\in J_h}|Y(t_k)-y_k|=\sup_{k=(h-1)N,\ldots,hN}|e_k|,\quad\text{ with $h\geq 0$.}
\end{equation*}
Similarly, let $E^X_{(h)}$ be the approximation error on the time interval $J_h$, that is,
\begin{equation*}\label{eq:ErrorIntervalEX}
E^X_{(h)}:=\sup_{t_k\in J_h}|X(t_k)-x_k|=\sup_{k=(h-1)N,\ldots,hN}|X(t_k)-x_k|,\quad\text{ with $h\geq0$.}
\end{equation*}

For $0\leq h\leq \left\lceil\frac{T-t_0}{\tau}\right\rceil$, we consider the following inequalities
\begin{align*}
\left(\mathcal{I}_{(h)}^Y\right):\quad \text{for all $p\in\left(1,\frac{2\,a\underline{\gamma}}{\sigma^2}\,\frac{1}{1+\frac{h-1}{2}}\right)$, there exists a constant  $C_{p,h}^Y$ such that } \|E^Y_{(h)}\|_p\leq C_{p,h}^Y\left(\frac{\tau}{N}\right)^{\frac{1}{2}},\\
 \intertext{and}
\left(\mathcal{I}_{(h)}^X\right):\quad \text{for all $p\in\left(1,\frac{2\,a\underline{\gamma}}{\sigma^2}\,\frac{1}{1+\frac{h-1}{2}}\right)$, there exists a constant  $C_{p,h}^X$ such that } \|E^X_{(h)}\|_p\leq C_{p,h}^X\left(\frac{\tau}{N}\right)^{\frac{1}{2}}.\\
\end{align*}
The idea is to prove the following inequalities chain
\begin{equation}\label{eq:Implication}
\left(\mathcal{I}_{(h)}^Y\right)\quad\Longrightarrow\quad\left(\mathcal{I}_{(h)}^X\right)
\quad\Longrightarrow\quad\left(\mathcal{I}_{(h+1)}^Y\right).
\end{equation}
Then the thesis is achieved, when we get inequality $\left(\mathcal{I}_{(h)}^X\right)$ for $h$ such that $T\in J_h$, i.e.,  $h=\lceil\frac{T-t_0}{\tau}\rceil$.\\

We start proving the first implication.\\
By the following equalities
\[|X(t_k)-x_k|=|Y^2(t_k)-y^2_k|=|Y(t_k)-y_k||Y(t_k)+y_k|,\]
we have that, for any $q>1$,
\begin{align*}
\left\|E^X_{(h)}\right\|_{p}\leq&\left\|E^Y_{(h)}\right\|_{pq}
\left\|\sup_{k\in\{(h-1)N,\ldots,hN\}}|Y(t_k)+y_k|\right\|_{p\frac{q}{q-1}}.\\
\intertext{Since $p\in\left(1,\frac{2\,a\underline{\gamma}}{\sigma^2}\,\frac{1}{1+\frac{h-1}{2}}\right)$, we can take a $q>1$ such that $pq\in\left(1,\frac{2\,a\underline{\gamma}}{\sigma^2}\,\frac{1}{1+\frac{h-1}{2}}\right)$, obtaining}
\left\|E^X_{(h)}\right\|_{p}\leq & C_{pq,h}^Y\left(\frac{\tau}{N}\right)^\frac{1}{2}\left(\left\|\sup_{k\in\{(h-1)N,\ldots,hN\}}|Y(t_k)|\right\|_{p\frac{q}{q-1}}+
\left\|\sup_{k\in\{(h-1)N,\ldots,hN\}}|y_k|\right\|_{p\frac{q}{q-1}}\right),\\
\end{align*}
and consequently, by Corollary~\ref{cor:meanY} and Proposition~\ref{prop:DNS3.4} (see conditions~\eqref{eq:IntCondP2} and~\eqref{eq:MBIS_prophp}, respectively) there exists a constant $C^X_{p,h}$ such that $\left(\mathcal{I}_{(h)}^X\right)$ holds.\\

Before proving the second implication in~\eqref{eq:Implication} we obtain a recursive formula  for the error sequence~$e_k$, defined in~\eqref{eq:ErrorSequence}. Taking into account the implicit discretization scheme in the form~\eqref{eq:AS_f} and the integral form of Eq.~\eqref{eq:diffProcessY} for $Y$, we have
\begin{align*}
e_0=&Y(t_0)-y_0,\\
e_{k+1}=&e_k+\left[f(t_{k+1},Y(t_{k+1}),
Y(t_{k+1-N}))-f(t_{k+1},y_{k+1}, y_{k+1-N})\right]
\frac{\tau}{N}
\\&-\int_{t_k}^{t_{k+1}}\big[f(t_{k+1}, {Y}(t_{k+1}), Y(t_{k+1} -\tau))-f(t,{Y}(t), Y(t-\tau))\big]\,dt,
\end{align*}
where $f(t,y,z)$ is defined in~\eqref{eq:def_f(t,y,z)}. Then, setting
\[
\Delta f_k:=\big[f(t_{k+1},Y(t_{k+1}),
Y(t_{k+1-N}))-f(t_{k+1},y_{k+1}, y_{k+1-N})\big],
\]
and
\[ \widetilde{f}(t,Y)=f(t,{Y}(t), Y(t-\tau)), \quad\text{and then}\quad \widetilde{f}(t_{k+1},Y)=f(t_{k+1}, {Y}(t_{k+1}), Y(t_{k+1} -\tau)),  \]
we have that
\begin{equation*}\label{eq:recursive-form-ek}
e_{k+1}=e_k+\Delta f_k\, \frac{\tau}{N} - \int_{t_k}^{t_{k+1}}\big[\widetilde{f}(t_{k+1}, Y)-\widetilde{f}(t,Y)\big]\,dt.
\end{equation*}

Multiplying both sides by $e_{k+1}$, we obtain
\begin{align*}
e^2_{k+1}=&e_{k+1}e_{k}+e_{k+1}\,\Delta f_k\, \frac{\tau}{N}
-e_{k+1}\int_{t_k}^{t_{k+1}}\big[\widetilde{f}(t_{k+1},Y)-\widetilde{f}(t,Y)\big]\,dt\\
\leq&
\frac{1}{2} e^2_{k+1} +\frac{1}{2} e^2_{k} +e_{k+1}\,\Delta f_k\, \frac{\tau}{N}
-e_{k+1}\int_{t_k}^{t_{k+1}}\big[\widetilde{f}(t_{k+1},Y)-\widetilde{f}(t,Y)\big]\,dt\\
=&\frac{1}{2} e^2_{k+1} +\frac{1}{2} e^2_{k} +e_{k+1} \,\Delta f_k\,  \frac{\tau}{N}+e_{k+1}r_k,
\end{align*}
where
\begin{equation}\label{eq:def_r_k}r_k:=-\int_{t_k}^{t_{k+1}}\big[\widetilde{f}(t_{k+1},Y)-\widetilde{f}(t,Y)\big]\,dt,\end{equation}
is the so-called local error.\\

By~\eqref{eq:FirstInequalityForf} in Lemma~\ref{lem:DisugualianzePerf}, we get that, for $n\geq hN$, i.e., such that $t_n \in[t_0+h\tau,t_0+(h+1)\tau]$
\begin{align*}
0\leq e_n^2=
2\sum_{k=hN}^{n-1}\left(\frac{e^2_{k+1}}{2}-\frac{e^2_k}{2}\right)+e_{hN}^2
\leq
\overline{b} \sum_{k=hN}^{n-1} \frac{Y^2(t_{k+1-N})- y^2_{k+1-N} }{Y(t_{k+1})} e_{k+1}   \,  \frac{\tau}{N}
+2\sum_{k=hN}^{n-1}e_{k+1}r_k+e_{hN}^2.
\end{align*}
Consequently, observing that
\begin{align*}
0\leq e_n^2
\leq&
\overline{b} \sum_{k=hN}^{n-1} \frac{\left|X(t_{k+1-N})- x_{k+1-N}\right|  }{Y(t_{k+1})}|e_{k+1}| \,  \frac{\tau}{N}
+2\sum_{k=hN}^{n-1}|e_{k+1}| |r_k|+e_{hN}^2
 \\
\leq&  \left( \overline{b} \sum_{k=hN}^{n-1} \frac{\left|X(t_{k+1-N})- x_{k+1-N} \right|}{Y(t_{k+1})}   \,  \frac{\tau}{N}
+2\sum_{k=hN}^{n-1} |r_k|\right) \, \sup_{hN\leq j\leq n} |e_{j}|+\sup_{hN\leq j\leq n} |e_{j}|\sup_{(h-1)N\leq j\leq hN} |e_{j}|,
\end{align*}
we obtain
\begin{align*}E^Y_{(h+1)}:=&\sup_{k=hN,\ldots, (h+1)N}|e_k|\leq \overline{b} \sum_{k=hN}^{(h+1)N-1} \frac{\left|X(t_{k+1-N})-x_{k+1-N} \right|}{X^\frac{1}{2}(t_{k+1})}   \frac{\tau}{N}+ 2 \sum_{k=hN}^{(h+1)N-1} \left|r_k\right|+E^Y_{(h)}\\
\leq&\overline{b}\tau \left|E^X_{(h)}\right|\,\frac{1}{N}\,\sum_{k=hN}^{(h+1)N-1}  \frac{1}{X^\frac{1}{2}(t_{k+1})}   + 2 \sum_{k=hN}^{(h+1)N-1} \left|r_k\right|+E^Y_{(h)}.
\end{align*}

Consequently, for any $\epsilon>0$, we obtain
\begin{align*}
\|E^Y_{(h+1)}\|_p\leq&\overline{b} \tau \,\frac{1}{N}\, \sum_{k=hN}^{(h+1)N-1}  \left\| E^X_{(h)}\ \frac{1}{X^{\frac{1}{2}}(t_{k+1})}\right\|_p\,
+ 2 \sum_{k=hN}^{(h+1)N-1}\left\|r_k\right\|_p+\|E^Y_{(h)}\|_p
\\
\leq&\overline{b} \tau \left\|E^X_{(h)}\right\|_{p(1+\epsilon)}\sup_{t\in[t_0,T]} \left\|\frac{1}{X(t)}\right\|^2_{p\frac{1+\epsilon}{2\epsilon}}\, + 2 \sum_{k=hN}^{(h+1)N-1}\left\|r_k\right\|_p+\|E^Y_{(h)}\|_p.
\end{align*}
Since $\overline{b}>0$, we need an upper bound for $\left\|\frac{1}{X(t)}\right\|^2_{p\frac{1+\epsilon}{2\epsilon}}$: By taking into account~\eqref{eq:comparison} in Proposition~\ref{prop:confronto},   inequality~\eqref{eq:exBDb}, for $p\in [1, \frac{2a\underline{\gamma}}{\sigma^2})$, together with equality~\eqref{eq:HKprop2}, for $p\geq  \frac{2a\underline{\gamma}}{\sigma^2}$ in Lemma~\ref{lem:exBD}, we have an upper bound only when $p\frac{1+\epsilon}{2\epsilon} < \frac{2a\underline{\gamma}}{\sigma^2}$. Therefore, taking into account also the assumed condition $\left(\mathcal{I}_{(h)}^X\right)$, we have to choose $\epsilon>0$ such that
\[\begin{cases}
&1<p(1+\epsilon)<\frac{2a\underline{\gamma}}{\sigma^2} \frac{1}{1+\frac{h-1}{2}} \\
&\frac{p}{2}\frac{(1+\epsilon)}{\epsilon}<\frac{2a\underline{\gamma}}{\sigma^2}.
\end{cases}\]
In particular, we can find an $\epsilon>0$  satisfying this system if and only if $p<\frac{2a\underline{\gamma}}{\sigma^2}\frac{1}{1+\frac{h}{2}}$; indeed
 necessarily the function $g_h(\epsilon):=\max \left((1+\epsilon)\frac{h+1}{2},\frac{1}{2}\frac{(1+\epsilon)}{\epsilon}\right)$ has to be  less then $\frac{2a\underline{\gamma}}{\sigma^2 \,p}$, and, since  $g(\epsilon)\geq g(\frac{1}{h+1})=1+\frac{h}{2}$, the best choice is $\epsilon=\frac{1}{h+1}$ and we need to assume that $p(1+\frac{h}{2})<\frac{2a\underline{\gamma}}{\sigma^2}$. With the latter choice of $\epsilon$, and using the assumption~\eqref{eq:HPappErr}   when $h=0$, while using  the first implication in~\eqref{eq:Implication} when $h\geq 1$, we get, with $C =\overline{b} \tau \,C_{p\frac{h+2}{h+1},h}^X$,
\begin{equation}\label{eq:supe_k}
\|E^Y_{(h+1)}\|_p\leq C \,\left(\frac{\tau}{N}\right)^{\frac{1}{2}} + 2 \sum_{k=hN}^{(h+1)N-1}\left\| r_k\right\|_p+C_{p,h}^Y\left(\frac{\tau}{N}\right)^{\frac{1}{2}}\quad\text{for all $p\in\left(1,\frac{2a\underline{\gamma}}{\sigma^2}\frac{1}{1+\frac{h}{2}}\right)$}.
\end{equation}

To achieve the second implication in~\eqref{eq:Implication}, and therefore the thesis,
it is sufficient to prove that
\[\text{for all } p\in \left(1,\frac{2a\underline{\gamma}}{\sigma^2} \right),\text{ there exists a constant $\Phi_p$ such that } \left\| r_k\right\|_p\leq\left(\frac{\tau}{N}\right)^{\frac{3}{2} }\Phi_p^{\frac{1}{p}}.\]
By \eqref{eq:SecondInequalityForf} in Lemma~\ref{lem:DisugualianzePerf} and by the definition \eqref{eq:def_r_k} of local error $r_k$, we have that
\[
|r_k|\leq\int_{t_k}^{t_{k+1}}\left|\widetilde{f}(t_{k+1},Y)-\widetilde{f}(t,Y)\right|\,dt\leq I_1+I_2+I_3,\]
where
\begin{align*}
&I_1:=\frac{1}{Y(t_{k+1})}\int_{t_k}^{t_{k+1}} | \underline{a}(t_{k+1})- \underline{a}(t)| \,dt,\\
&I_2:=\int_{t_k}^{t_{k+1}}|Y(t_{k+1})-Y(t)|\left(\frac{ \underline{a}(t)+ \overline{b}  Y^2(t-\tau) }{Y(t_{k+1})Y(t)}+\overline{a}\right)\,dt,\\
&I_3:=\overline{b}  \int_{t_k}^{t_{k+1}} \frac{1}{Y(t_{k+1})}\,  |Y^2(t_{k+1}-\tau) - Y^2(t -\tau)| \,dt\\
&\phantom{I_3}=\overline{b}  \int_{t_k}^{t_{k+1}} \frac{1}{Y(t_{k+1})}\,  |X(t_{k+1}-\tau) - X(t -\tau)| \,dt.\\
\end{align*}

Consequently, an upper bound for the mean of the local error $r_k$ is given by
\begin{equation}\label{eq:meanr_k}
\mathbb{E}\left[|r_k|\right]\leq(\mathbb{E}\left[|r_k|^p\right])^{\frac{1}{p}}
=\|r_k\|_p\leq\|I_1\|_p+\|I_2\|_p+\|I_3\|_p.
\end{equation}

Now, we determine upper bounds for $\|I_1\|_p$, $\|I_2\|_p$ and $\|I_3\|_p$.

Since the function $\gamma(t)$ satisfies the condition \eqref{eq:holdercond}, and by point \textbf{(ii)} in Assumptions~\ref{ass:InitialValue}, we get
\begin{align*}
\|I_1\|_p^p\leq
&\left(\frac{aL}{2}\right)^p\left(\int_{t_k}^{t_{k+1}}|t_{k+1}-t|^{\frac{1}{2}}\,dt\right)^p\mathbb{E}\left[\frac{1}{Y^p(t_{k+1})}\right]
\leq
\left(\frac{aL}{2}\right)^p\left(\frac{2}{3}\left(\frac{\tau}{N}\right)^{\frac{3}{2}}\right)^p\mathbb{E}\left[\frac{1}{Y^p(t_{k+1})}\right]\\
\leq
&\left(\frac{aL}{3}\right)^p\left(\frac{\tau}{N}\right)^{\frac{3}{2}p}\mathbb{E}\left[\frac{1}{Y^p(t_{k+1})}\right]
\leq \left(\frac{aL}{3}\right)^p\left(\frac{\tau}{N}\right)^{\frac{3}{2}p}\sup_{u\in[t_0,T]}
\mathbb{E}\left[\frac{1}{X^{\frac{p}{2}}(u)}\right]<+\infty, \quad \text{for }\, p<\frac{4a\underline{\gamma}}{\sigma^2}.
\end{align*}

Using H\"{o}lder inequality twice with $p$ and $q>1$ and $q^\prime=\tfrac{q}{q-1}$, we have
\begin{align*}
\|I_2\|_p^p\leq
&\left(\frac{\tau}{N}\right)^{p-1}\int_{t_k}^{t_{k+1}}
\mathbb{E}\left[\left(|Y(t_{k+1})-Y(t)|\left(\frac{\underline{a}(t)+ \overline{b}  Y^2(t-\tau) }{Y(t_{k+1})Y(t)}+\overline{a}\right)\right)^p\right]\,dt\\
\leq
&\left(\frac{\tau}{N}\right)^{p-1}\int_{t_k}^{t_{k+1}}\mathbb{E}\left[|Y(t_{k+1})-Y(t)|^{pq}\right]^{\frac{1}{q}}
\mathbb{E}\left[\left(\frac{\underline{a}(t)+ \overline{b}  Y^2(t-\tau) }{Y(t_{k+1})Y(t)}+\overline{a}\right)^{pq^\prime}\right]^{\frac{1}{q^\prime}}\, dt.\\
\intertext{By point $\emph{1}.$ of Corollary~\ref{cor:meanY}, since $|t_{k+1}-t|\leq \frac{\tau}{N}$, we get}
\|I_2\|_p^p\leq
&c_{pq,1}\left(\frac{\tau}{N}\right)^{p-1}\int_{t_k}^{t_{k+1}}\left(|t_{k+1}-t|^{\frac{pq}{2}}\right)^{\frac{1}{q}}
\mathbb{E}\left[\left(\frac{\underline{a}(t)+ \overline{b}  Y^2(t-\tau) }{Y(t_{k+1})Y(t)}+\overline{a}\right)^{pq^\prime}\right]^{\frac{1}{q^\prime}}\,dt\\
\leq&c_{pq,1}\left(\frac{\tau}{N}\right)^{\frac{3}{2}p-1}\int_{t_k}^{t_{k+1}}
\mathbb{E}\left[\left(\frac{\underline{a}(t)+ \overline{b}  Y^2(t-\tau) }{Y(t_{k+1})Y(t)}+\overline{a}\right)^{pq^\prime}\right]^{\frac{1}{q^\prime}}\,dt.\\
\end{align*}
Applying H\"{o}lder inequality with $\alpha>1$, and $\alpha^\prime=\frac{\alpha}{\alpha-1}$, we get an upper bound for the expectation inside the integral; indeed,
\begin{align*}
&\mathbb{E}\left[\left(\tfrac{\underline{a}(t)+\overline{b}Y^2(t-\tau)}{Y(t_{k+1})Y(t)}\right)^{pq^\prime}\right]\leq
\left(\mathbb{E}\left[(\underline{a}(t)+\overline{b}Y^2(t-\tau))^{\alpha pq^\prime}\right]\right)^{\frac{1}{\alpha}}
\left(\mathbb{E}\left[\left(\tfrac{1}{Y(t_{k+1})Y(t)}\right)^{\alpha^\prime pq^\prime}\right]\right)^{\frac{1}{\alpha^\prime}}\\
&\leq \left(\mathbb{E}\left[\big(\underline{a}(t)+\overline{b}Y^2(t-\tau)\big)^{\alpha pq^\prime}\right]\right)^{\frac{1}{\alpha}}
\left(\tfrac{1}{2}\left(\mathbb{E}\left[\left(\tfrac{1}{Y(t_{k+1})}\right)^{2\alpha^\prime pq^\prime}\right]+\mathbb{E}\left[\left(\tfrac{1}{Y(t)}\right)^{2\alpha^\prime pq^\prime}\right]\right)\right)^{\frac{1}{\alpha^\prime}},\\
\intertext{since $X(t)=Y^2(t)$, by Proposition~\ref{prop:confronto} and by point 1. of Proposition~\ref{prop:p-thmoment}, we have}
&\leq\left(2^{\alpha pq^\prime-1}\left((\underline{a}^*)^{\alpha pq^\prime}+\overline{b}^{\alpha pq^\prime}K_{\alpha p q^\prime}\right)\right)^{\frac{1}{\alpha}}
\left(\tfrac{1}{2}\left(2\sup_{t\in[t_0,T]}
\mathbb{E}\left[\left(\tfrac{1}{\underline{X}(t)}\right)^{\alpha^\prime pq^\prime}\right]\right)\right)^\frac{1}{\alpha^\prime}.
\end{align*}
Since $p<\frac{2a\underline{\gamma}}{\sigma^2}$, we can choose $\alpha^\prime>1$ and $q^\prime>1$ such that
\[\sup_{t\in[t_0,T]}
\mathbb{E}\left[\left(\tfrac{1}{\underline{X}(t)}\right)^{\alpha^\prime pq^\prime}\right]<+\infty,\]
so that, by the point \textbf{(i)} of Assumptions~\ref{ass:InitialValue} and  Lemma~\ref{lem:exBD}, we have that
\begin{align*}
\text{for all $p\in \left(1,\frac{2a\underline{\gamma}}{\sigma^2}  \right)$ there exist constants $\Psi_p$ such that, } \|I_2\|_p^p\leq&\Psi_p\left(\frac{\tau}{N}\right)^{\frac{3}{2}p}.
\end{align*}

Using H\"{o}lder inequality twice and for any $\nu>1$, we have
\begin{align*}
\|I_3\|_p^p\leq&\overline{b}^p\left(\frac{\tau}{N}\right)^{p-1}\int_{t_k}^{t_{k+1}}  \mathbb{E}\left[ \frac{1}{Y^p(t_{k+1})}\, |X(t_{k+1}-\tau) - X(t -\tau)|^p \right] \,dt\\
\leq&\overline{b}^p\left(\frac{\tau}{N}\right)^{p-1}\int_{t_k}^{t_{k+1}}\left(\mathbb{E}\left[\frac{1}{Y^{p\nu}
(t_{k+1})}\right]\right)^{\frac{1}{\nu}}
\left(\mathbb{E}\left[|X(t_{k+1}-\tau) - X(t -\tau)|^{p\nu^\prime} \right]\right)^{\frac{1}{\nu^\prime}}dt.
\intertext{By point $\emph{2}.$ of Proposition~\ref{prop:p-thmoment}, since $|t_{k+1}-t|\leq \frac{\tau}{N}$, we get}
\|I_3\|_p^p\leq&\overline{b}^p\left(\frac{\tau}{N}\right)^{\frac{3}{2}p}C_{p\nu^\prime}
\left(\mathbb{E}\left[\left(\frac{1}{X(t_{k+1})}\right)^{\frac{p\nu}{2}}\right]\right)^{\frac{1}{\nu}}<+\infty\quad\text{for $p\nu<\frac{4a\underline{\gamma}}{\sigma^2}$}.
\end{align*}

Summarizing, by~\eqref{eq:meanr_k}, we have that
 for every $ p\in \left(1,\frac{2a\underline{\gamma}}{\sigma^2}  \right)$, there exists a constant $\Phi_p$ such that $$ \|r_k\|_p \leq \left(\|I_1\|_p +\|I_2\|_p +\|I_3\|_p \right)\leq   \left(\frac{\tau}{N}\right)^{\frac{3}{2} }\Phi_p^{\frac{1}{p}}.$$

Finally, by \eqref{eq:supe_k} we obtain the second implication in~\eqref{eq:Implication},
and we are done.
\end{proof}

\begin{remark}\label{oss:prop:EB_IS_Y}${}$\\
From the proof of the previous Proposition~\ref{prop:EB_IS_Y}, it is clear that, when $b=0$, one can take any value for $\tau$ and in particular one can take $\tau=T-t_0$, so that $\Delta=\frac{T-t_0}{N}$. Furthermore, the assumption on $E^Y_{(0)}$ reduces to~$\big(\mathbb{E}[|Y(t_0)-y_0|^p]\big)^{\frac{1}{p}}\leq C^Y_{p,0} \Delta^{\frac{1}{2}}$, one can consider only  the first (and unique) interval $[t_0, T]$ of size $\tau=T-t_0$. Finally, since $\overline{b}=\frac{b}{2}=0$,  an upper bound for $\left\|\frac{1}{X(t)}\right\|^2_{p\frac{1+\epsilon}{2\epsilon}}$ is not necessary, and  one can obtain the result of~Proposition 3.3.\@ in~\cite{dereich2011euler}. (Actually we get a slight extension, since we do not need to assume  $y_0=Y(t_0)$). The same considerations hold for Theorem~\ref{thm:Teorema1} in the following section.
\end{remark}
\section{Error Bound for the piecewise-linear interpolation $\widehat{X}(t)$}\label{sec:EBDISRES}
In this section we prove the main result of this paper, Theorem~\ref{thm:Teorema1}, i.e., a strong convergence result, extending Theorem $1.1$ in~\cite{dereich2011euler}.
\\

Before stating it we need  to define an intermediate approximation of $X(t)$ that will be used in the proof.\\
Denote by $Z(t)$ the piecewise linear interpolation of the fixed delay CIR process  $X(t)$ defined in~\eqref{eq:CIRdelayX} with initial segment process~\eqref{eq:CIRdelayX-InitialSegment} , with discretization step $\Delta=\Delta t_k=\frac{\tau}{N}$; that is
\begin{equation}\label{eq:PLI}
Z(t)= X(t_k)+ (t-t_k)\,\frac{X(t_{k+1})-X(t_k)}{t_{k+1}-t_k}, \qquad t\in [t_k, t_{k+1}].
\end{equation}

\begin{theorem}\label{thm:Teorema1}${}$\\
Assume the same conditions of Proposition~\ref{prop:EB_IS_Y}, but with condition~\eqref{eq:HPappErr} substituted by the stronger assumption
\begin{equation}\label{eq:HPappErr-X-0}
\left(\textstyle{\mathbb{E}\left[\sup_{t_k \in[t_0-\tau,t_0]}|X_0(t_k)-x_k|^p\right]}\right)^{\frac{1}{p}}\leq C^X_{p,0} \, \Delta^{\frac{1}{2}}, \qquad \text{for all $p\geq 1$.}
\end{equation}
Then,
for all $p\in\left[1,\tfrac{2a\underline{\gamma}}{\sigma^2}\tfrac{2}{1+\big\lceil\tfrac{T-t_0}{\tau}\big\rceil}\right)$, there exists a constant $\Theta_p>0$ such that
\[
\left(\mathbb{E}\left[\textstyle{\sup_{t\in[t_0,T]}} |X(t)-\widehat{X}(t)|^p \right] \right)^{\frac{1}{p}}\leq \Theta_p \left( \Delta \left| \log \left( \Delta\right) \right|\right)^{\frac{1}{2}}.
\]
\end{theorem}
\begin{proof}${}$\\
Let $\widehat{X}(t)$ be the linear interpolation of $x_k=y_k^2$ defined in \eqref{eq:LinearInterpolationDNSScheme} and $Z(t)$ be the linear interpolation of $X(t_k)$ defined in \eqref{eq:PLI}. We have
\begin{align*}
\sup_{t\in [t_0,T]}|X(t)-\widehat{X}(t)|&\leq   \sup_{t\in [t_0,T]}|X(t)-Z(t)|+\sup_{t\in [t_0,T]}|Z(t)-\widehat{X}(t)|
\intertext{and therefore}
\left(\mathbb{E}\left[\textstyle{\sup_{t\in[t_0,T]}|X(t)-\widehat{X}(t)|^p}\right]\right)^{\frac{1}{p}}
&
\leq
\left(\mathbb{E}\left[\textstyle{\sup_{t\in[t_0,T]}|X(t)-Z(t)|^p}\right]\right)^{\frac{1}{p}}+
\left(\mathbb{E}\left[\textstyle{\sup_{t\in[t_0,T]}|Z(t)-\widehat{X}(t)|^p}\right]\right)^{\frac{1}{p}}
\\
&\leq \textstyle{ \left(\mathbb{E}\left[\textstyle{\sup_{t\in[t_0,T]}|X(t)-Z(t)|^p}\right]\right)^{\frac{1}{p}}  +
\left(\mathbb{E}\left[\sup_{k\,:\,t_k\in[t_0,T]}|X(t_k)- x_k|^p\right]\right)^{\frac{1}{p}}}.
\end{align*}
 Since $|\sqrt{x}-\sqrt{y}|\leq \sqrt{|x-y|}$, condition~\eqref{eq:HPappErr-X-0} clearly implies condition~\eqref{eq:HPappErr}; therefore
 we can apply Proposition~\ref{prop:EB_IS_Y}, and get that
 $\left(\mathbb{E}\left[\sup_{k\,:\,t_k\in[t_0,T]}|X(t_k)- x_k|^p\right]\right)^{\frac{1}{p}}\leq C^X \Delta^{\frac{1}{2}}$.
 \\ The thesis is then achieved by the inequality
 $$
 \left(\mathbb{E}\left[\textstyle{\sup_{t\in[t_0,T]}|X(t)-Z(t)|^p}\right]\right)^{\frac{1}{p}}\leq  c_p \left( \Delta \left| \log \left( \Delta\right) \right|\right)^{\frac{1}{2}},
 $$
  which is proved in the subsequent Lemma~\ref{lem:EBDI}.
\end{proof}

\begin{lemma}\label{lem:EBDI}${}$\\
Under the same hypotheses of Theorem~\ref{thm:Teorema1}, let $Z(t)$ be the piecewise linear interpolation~\eqref{eq:PLI} of the fixed delay CIR process.
Then,
 for all $p\geq 1$, we have
\begin{equation*}\label{eq:EBDIq}
\left(\mathbb{E}\left[\textstyle{\sup_{t\in[t_0,T]} }|X(t)-Z(t)|^p\right]\right)^\frac{1}{p}\leq c_p\left(\Delta\left|\log\left(\Delta\right)
         \right|\right)^\frac{1}{2}.
\end{equation*}
\end{lemma}
\begin{proof}${}$\\
The assertion follows by observing that, for any $t\in [t_k, t_{k+1}]$, $k\geq 0$,
$$
Z(t)=\lambda_k X(t_k) + (1-\lambda_k) X(t_{k+1})\quad  \text{and} \quad X(t)=\lambda_k X(t) + (1-\lambda_k) X(t),
$$
with $\lambda_k= \frac{t_{k+1}-t}{t_{k+1}-t_k}$, and therefore,
\begin{align*}\sup_{t\in[t_0,T]}|X(t)-Z(t)|\leq&\sup_{\overset{t,s\in [t_0,T]:}{|t-s|\leq \frac{\tau}{N}}}|X(t)-X(s)|=:w_X(\tfrac{\tau}{N};[t_0,T]).
\end{align*}
Then, to get the thesis  it is sufficient to recall that  condition \textbf{(i)}~in Assumptions~\ref{ass:InitialValue} implies~\eqref{eq:MODULUS-IntCondEA_HP2-unif} on the $p$-moments of the modulus of continuity $w_X(\delta;[t_0,T])$  (see  \mbox{point~$\emph{4}.$} of Proposition~\ref{prop:p-thmoment}).
\end{proof}

\section{Appendix: Some Results on the classical CIR Model}\label{app:BSP}
In this appendix, we focus our interests on some results related to CIR process with constant long-term value, given by \begin{equation}\label{eq:CIRprocessClassic}
\begin{cases}
d\underline{X}(t)=&a\left(\underline{\gamma}-\underline{X}(t)\right)dt+\sigma\sqrt{\underline{X}(t)}dW(t),\\
\underline{X}(t_0)=&\underline{X}_0,\end{cases}\end{equation}
where $a$, $\underline{\gamma}$ and $\sigma$ are positive constants, and $\underline{X}_0$ is a positive random variable.\\

With the following Lemma~\ref{lem:exBD}, we prove a  generalization of Lemma $A.1$ of~\cite{bossy2007efficient}).

\begin{lemma}\label{lem:exBD}${}$\\
Consider the process $\underline{X}(t)$ defined by~\eqref{eq:CIRprocessClassic}. Assume that  $\frac{2a\underline{\gamma}}{\sigma^2}>1$.
 If
$p\geq \frac{2a\underline{\gamma}}{\sigma^2},$
then
\begin{equation}\label{eq:HKprop2}
\mathbb{E}\left[\frac{1}{\underline{X}^p(t)}\right]=+\infty, \quad \text{for
all $t\geq t_0$},
\end{equation}
while, if
\[
1 \leq p < \frac{2a\underline{\gamma}}{\sigma^2}
\qquad \text{and} \qquad
\mathbb{E}\left[\frac{1}{\underline{X}^p_0}\right]<+\infty,\]
then, there exists a constant $L_p$ such that, for any $t\geq t_0$
\begin{equation}\label{eq:exBDb}\mathbb{E}\left[\frac{1}{\underline{X}^p(t)}\right]\leq L_p\mathit{e}^{apt}\mathbb{E}\left[\frac{1}{\underline{X}^p_0}\right].\end{equation}
Moreover,
\begin{enumerate}
\item if $\frac{2a\underline{\gamma}}{\sigma^2}\geq 2$ and  $p$ is such that $1\leq p\leq \frac{2a\underline{\gamma}}{\sigma^2}-1$, then we can take
\begin{equation*}\label{eq:LpA}L_p=1;\end{equation*}
\item if $\frac{2a\underline{\gamma}}{\sigma^2}\geq 2$ and  $p$ is such that $\frac{2a\underline{\gamma}}{\sigma^2}-1<p<\frac{2a\underline{\gamma}}{\sigma^2}$,  then we can take
\begin{equation}\label{eq:LpB}
L_p=\frac{1}{2^{\frac{2a\underline{\gamma}}{\sigma^2}-p-1}}+
\frac{2^{2p-\frac{2a\underline{\gamma}}{\sigma^2}}p^p\mathit{e}^{-p}}
{\Gamma(p)\left(\frac{2a\underline{\gamma}}{\sigma^2}-p\right)}
= 2^{p+1-\frac{2a\underline{\gamma}}{\sigma^2}} \left(1+
\frac{2^{p-1}p^p\mathit{e}^{-p}}
{\Gamma(p)\big(\frac{2a\underline{\gamma}}{\sigma^2}-p\big)}\right),
\end{equation}
\item if $\frac{2a\underline{\gamma}}{\sigma^2}< 2$, we can take $L_p$ as in \eqref{eq:LpB}.
\end{enumerate}
\end{lemma}
\begin{remark}\label{oss:exBD}${}$\\
The main difference with Lemma $A.1$ of Bossy and Diop is that in~\cite{bossy2007efficient}, the authors deal only with cases $1.$ and $2.$, without giving an explicit bound to the constants. Moreover, as explained in Remark~\ref{oss:bossy-diop}, one could get the constant $L_p=1$ also under different conditions on $p$ and $\frac{2a\underline{\gamma}}{\sigma^2}$.
Finally, the initial condition~$\underline{X}(t_0)$ is a random variable $\underline{X}_0$, while in~\cite{bossy2007efficient}, the initial condition is a constant.
\end{remark}
\begin{proof}[Proof of Lemma~\ref{lem:exBD}]${}$\\
By the successive conditioning property, we have that
\begin{equation}\label{eq:exBD1}
\mathbb{E}\left[\frac{1}{\underline{X}^p(t)}\right]
=\mathbb{E}\left[\mathbb{E}_{t_0,\underline{X}(t_0)}\left[\frac{1}{\underline{X}^p(t)}\right]\right].
\end{equation}
Then \eqref{eq:HKprop2} is a direct consequence of Theorem 3.1 in~\cite{HurdKuz}.
\\

By the definition of the Gamma function  $\Gamma(p)=\int_0^\infty\,t^{p-1}\mathit{e}^{-t}dt$, and the change of variable $t=ux$, one obtains that $x^{-p}=\frac{1}{\Gamma(p)}\int_0^\infty\,u^{p-1}\mathit{e}^{-xu}du$, and then the following representation holds
\[
\mathbb{E}_{t_0,x}\left[\frac{1}{\underline{X}^p(t)}\right]
=\frac{1}{\Gamma(p)}\int_0^\infty\,u^{p-1}\mathbb{E}_{t_0,x}\left[\mathit{e}^{-u\underline{X}(t)}\right]du.
\]
Since the Laplace transform for CIR model is known (see, e.g.,~\cite{lamberton_lapeyre}), we have that
\[\mathbb{E}_{t_0,x}\left[\frac{1}{\underline{X}^p(t)}\right]
=\frac{1}{\Gamma(p)}\int_0^\infty\,u^{p-1}\frac{1}{\left(2uL(t)+1\right)^{\frac{2a\underline{\gamma}}{\sigma^2}}}
\mathit{e}^{-\frac{uL(t)\zeta(t,x)}{2uL(t)+1}}du,\]
where
\[L(t)=\frac{\sigma^2}{4a}\left(1-\mathit{e}^{-at}\right)\quad\text{and}\quad\zeta(t,x)=\frac{x\mathit{e}^{-at}}{L(t)}.\]
By changing the variable $y=\frac{2uL(t)}{2uL(t)+1}\frac{\zeta(t,x)}{2}$, we get that
\begin{equation}\label{eq:exBD2}
\mathbb{E}_{t_0,x}\left[\frac{1}{\underline{X}^p(t)}\right]
=\frac{1}{\Gamma(p)}\frac{\mathit{e}^{apt}}{x^p}\int_0^{\frac{\zeta(t,x)}{2}}\,y^{p-1}
\left(1-\tfrac{2}{\zeta(t,x)}y\right)^{\frac{2a\underline{\gamma}}{\sigma^2}-p-1}
\mathit{e}^{-y}dy.\end{equation}
\textbf{\emph{1}.}\;
Since $\frac{2a\underline{\gamma}}{\sigma^2}-p-1\geq 0$,  $0<\left(1-\tfrac{2}{\zeta(t,x)}y\right)^{\frac{2a\underline{\gamma}}{\sigma^2}-p-1}\leq1$ for  every $y\in\left(0,\frac{\zeta(t,x)}{2}\right)$, and  consequently we obtain an upper bound for the integral in \eqref{eq:exBD2} as follows
      \[\int_0^{\frac{\zeta(t,x)}{2}}\,y^{p-1}
\left(1-\tfrac{2}{\zeta(t,x)}y\right)^{\frac{2a\underline{\gamma}}{\sigma^2}-p-1}
\mathit{e}^{-y}dy<\int_0^{\frac{\zeta(t,x)}{2}}\,y^{p-1}\mathit{e}^{-y}dy\leq
\int_0^{\infty}\,y^{p-1}
\mathit{e}^{-y}dy=\Gamma(p).\]
Therefore, we have that
\[\mathbb{E}_{t_0,x}\left[\frac{1}{\underline{X}^p(t)}\right]\leq\frac{\mathit{e}^{apt}}{x^p},\]
and, by \eqref{eq:exBD1}, we obtain the inequality~\eqref{eq:exBDb}, with $L_p=1$.\\\\
\textbf{\emph{2}.~and \emph{3}.}\;
  When either $\frac{2a\underline{\gamma}}{\sigma^2}\geq 2$ and  $\frac{2a\underline{\gamma}}{\sigma^2}-p-1< 0$, or   $\frac{2a\underline{\gamma}}{\sigma^2}< 2$, since $1\leq p <\frac{2a\underline{\gamma}}{\sigma^2}$ we are in the case\footnote{Setting $g=\frac{2a\underline{\gamma}}{\sigma^2}$, first of all we observe that, since $p<g$, i.e., $g-p>0$  then clearly $-1<g-p-1$.\\
       When  $g\geq 2$, and $g-p-1<0$ then \eqref{eq:exBD4} immediately follows.
      Similarly, when $1<g<2$ and $1\leq p<g$ then  $g-p-1<0$, indeed $g-p-1<2-p-1=1-p\leq 0$, and therefore, again \eqref{eq:exBD4} immediately follows.  }
       \begin{equation}\label{eq:exBD4}
       -1< \frac{2a\underline{\gamma}}{\sigma^2} -p-1<0, \qquad \text{and} \qquad p-1\geq 0.
       \end{equation}
      We divide the integral in~\eqref{eq:exBD2}
      into two integrals on the  subintervals $\left(0,\frac{\zeta(t,x)}{4}\right)$ and $\left(\frac{\zeta(t,x)}{4},\frac{\zeta(t,x)}{2}\right)$.

Since $-1<\frac{2a\underline{\gamma}}{\sigma^2}-p-1<0$, for $y\in\left(0,\frac{\zeta(t,x)}{4}\right)$, we have that \[1<\left(1-\tfrac{2}{\zeta(t,x)}y\right)^{\frac{2a\underline{\gamma}}{\sigma^2}-p-1}
<\left(\frac{1}{2}\right)^{\frac{2a\underline{\gamma}}{\sigma^2}-p-1}.\]
Consequently, we obtain the following upper bound for the  integral on the first interval
\begin{align*}
&\int_0^{\frac{\zeta(t,x)}{4}}\,y^{p-1}
\left(1-\tfrac{2}{\zeta(t,x)}y\right)^{\frac{2a\underline{\gamma}}{\sigma^2}-p-1}
\mathit{e}^{-y}dy< \int_0^{\frac{\zeta(t,x)}{4}}\,y^{p-1}
\left(\tfrac{1}{2}\right)^{\frac{2a\underline{\gamma}}{\sigma^2}-p-1}
\mathit{e}^{-y}dy
\\&< \int_0^{\infty}\,y^{p-1}
\left(\tfrac{1}{2}\right)^{\frac{2a\underline{\gamma}}{\sigma^2}-p-1}\mathit{e}^{-y}dy
=\left(\tfrac{1}{2}\right)^{\frac{2a\underline{\gamma}}{\sigma^2}-p-1}\Gamma(p).
\intertext{For the integral on the second  interval, taking into account that $p-1\geq 0$, we have that}
&\int_{\frac{\zeta(t,x)}{4}}^{\frac{\zeta(t,x)}{2}}\,y^{p-1}
\left(1-\tfrac{2}{\zeta(t,x)}y\right)^{\frac{2a\underline{\gamma}}{\sigma^2}-p-1}
\mathit{e}^{-y}dy<\left(\tfrac{\zeta(t,x)}{2}\right)^{p-1}\mathit{e}^{-\frac{\zeta(t,x)}{4}}
\int_{\frac{\zeta(t,x)}{4}}^{\frac{\zeta(t,x)}{2}}\,
\left(1-\tfrac{2}{\zeta(t,x)}y\right)^{\frac{2a\underline{\gamma}}{\sigma^2}-p-1}dy
\\
&=\left(\tfrac{\zeta(t,x)}{2}\right)^{p}\mathit{e}^{-\frac{\zeta(t,x)}{4}}\, \int_{\frac{1}{2}}^{1}\,
\left(1-z\right)^{\frac{2a\underline{\gamma}}{\sigma^2}-p-1}dz
= \left(\tfrac{\zeta(t,x)}{2}\right)^{p}\mathit{e}^{-\frac{\zeta(t,x)}{4}}
\frac{1}{\left(\frac{2a\underline{\gamma}}{\sigma^p}-p\right)}
\big(\tfrac{1}{2}\big)^{\frac{2a\underline{\gamma}}{\sigma^p}-p}
\\&<(2p)^p\mathit{e}^{-p}\frac{1}{\left(\frac{2a\underline{\gamma}}{\sigma^p}-p\right)}
\big(\tfrac{1}{2}\big)^{\frac{2a\underline{\gamma}}{\sigma^p}-p},
\end{align*}
since the function $t^p \mathit{e}^{-\frac{t}{2}}$ is uniformly bounded by~$(2p)^p\mathit{e}^{-p}$ in $(0,\infty)$.
Therefore, by~\eqref{eq:exBD2}, we have that
\[\mathbb{E}_{t_0,x}\left[\frac{1}{\underline{X}^p(t)}\right]\leq L_p\frac{\mathit{e}^{apt}}{x^p},\]
where
\[
L_p=\frac{1}{2^{\frac{2a\underline{\gamma}}{\sigma^2}-p-1}}+
\frac{2^{2p-\frac{2a\underline{\gamma}}{\sigma^2}}p^p\mathit{e}^{-p}}
{\Gamma(p)\left(\frac{2a\underline{\gamma}}{\sigma^2}-p\right)},\]
and, by \eqref{eq:exBD1}, we obtain the inequality~\eqref{eq:exBDb}.
\end{proof}

\begin{remark}\label{oss:bossy-diop}${}$\\
We observe that the inequality~\eqref{eq:exBDb}, with $L_p=1$, holds also in the case~$0<p<1$, and $p\leq \frac{2a\underline{\gamma}}{\sigma^2}-1$.
Indeed the proof  of point \emph{1.} can be repeated unchanged.  Moreover the condition $\frac{2a\underline{\gamma}}{\sigma^2}\geq 2$ is not necessary (the latter condition  is necessary since we take $p\geq 1$); for instance, if $p=\frac{1}{2}$, then \eqref{eq:exBDb} holds, with $L_p=1$, under the condition $\frac{3}{2}\leq \frac{2a\underline{\gamma}}{\sigma^2}$.
\end{remark}

We end this section using  Lemma $A.2$ in~\cite{bossy2007efficient}
as in Lemma 3.1 of~\cite{dereich2011euler} (see Lemma~\ref{prop:Lemma3.1DNS}), the only difference being in that our initial condition is random.

\begin{lemma}\label{prop:Lemma3.1DNS}${}$\\
Let $\underline{X}(t)$ be the classical CIR process defined in \eqref{eq:CIRprocessClassic}.
Assume that $\sigma^2< 2a \underline{\gamma}$, $T >t_0$ and
\begin{equation}\label{eq:UX0moments}
\mathbb{E}\left[\frac{1}{\underline{X}_0^{\nu/2}}\right]<+\infty,\quad \text{where $\nu=\frac{2a\underline{\gamma}}{\sigma^2}-1.$}
\end{equation}
Then, for any $q > 0$, there exists a constant $C=C(a,\underline{\gamma},\sigma,\nu, q,T-t_0)$ such that
\[\mathbb{E}\left[\left(\int_{t_0}^T\frac{1}{\underline{X}(t)}dt\right)^q\right]\leq  C \left( 1+ \mathbb{E}\left[ \frac{1}{\underline{X}^{\frac{1}{2}\nu}_0}\right] \right)<+\infty.
\]
\end{lemma}
\begin{proof}${}$\\
When the initial condition $\underline{X}_0$ is deterministic and equal to $x_0$, Lemma $A.2$ in~\cite{bossy2007efficient} guarantees that  there exists a constant $c=c(a,\gamma,\sigma,\nu,T-t_0)$ such that
\begin{equation*}
\sup_{t\in[t_0,T]}\mathbb{E}\left[\mathit{e}^{\frac{\nu^2\sigma^2}{8}\int_{t_0}^t\frac{1}{\underline{X}(s)}\,ds}\right]\leq
c\left(1+x^{-\frac{\nu}{2}}_0\right),
\end{equation*}
and condition~\eqref{eq:UX0moments} allows to extend it to random initial conditions.
\end{proof}


\end{document}